\documentclass[12pt,a4paper]{amsart}

\usepackage{euscript,amsfonts,amssymb,amsmath,amscd,amsthm,enumerate,hyperref}

\usepackage{comment}

\usepackage{graphicx}

\usepackage{color}

\usepackage[margin=1.1in]{geometry}
\newtheorem{theorem}{Theorem} %[section]
\newtheorem*{theorem*}{Theorem}
\newtheorem{lemma}[theorem]{Lemma}

\newtheorem{corollary}[theorem]{Corollary}

\theoremstyle{remark}
\newtheorem{remark}[theorem]{Remark}

\newcommand{\E}{\mathbb E}

\newcommand{\ii}{{\mathbf i}}

\newcommand{\eps}{\varepsilon}

\renewcommand{\t}{\mathbf{t}}

\newcommand{\C}{\mathcal C}

\newcommand{\1}{\mathbf 1}

\numberwithin{equation}{section} \numberwithin{theorem}{section}

\newcommand{\T}{\mathbf T}

\newcommand{\Tc}{\overline{\mathbf T}}

\newcommand{\vv}{\mathfrak v}

\DeclareMathOperator{\arcsh}{arcsh}

\title[Bulk universality for lozenge tilings]{Bulk universality for random lozenge tilings near
straight boundaries and for tensor products}

\author{Vadim Gorin}

\address[Vadim Gorin]{Department of Mathematics, Massachusetts Institute of Technology, Cambridge, MA, USA, and
 Institute for Information Transmission Problems of Russian Academy of Sciences, Moscow, Russia. E-mail: vadicgor@gmail.com}

\begin{document}

\maketitle

\begin{abstract}
 We prove that the asymptotic of the bulk local statistics in models of random lozenge tilings is universal
 in the vicinity of straight boundaries of the tiled domains. The result applies to uniformly
 random lozenge tilings of large polygonal domains on triangular lattice and to the probability measures describing the
 decomposition in Gelfand--Tsetlin bases of tensor products of representations of unitary groups.
 In a weaker form our theorem also applies to random domino tilings.
\end{abstract}

\section{Introduction}

\subsection{Overview}

This article is about random \emph{lozenge tilings}, which are tilings of domains on
the regular triangular grid by rhombuses of three types (see Figures \ref{Fig_Hex},
\ref{Fig_trap}, \ref{Fig_covered_domains}, \ref{Fig_not_covered_domain},
\ref{Fig_GT_tiling}), and which can be identified with dimers, discrete stepped
surfaces, $3d$ Young diagrams, see e.g.\ \cite{Kenyon_lectures}, \cite[Section
2]{BP}. In more details, we are interested in the asymptotic behavior of \emph{Gibbs
measures} on tilings, which means that tilings in each subdomain are always
conditionally uniformly distributed given boundary conditions, i.e.\ positions of
lozenges surrounding this domain. There are several ways to produce such Gibbs
measures. From the point of view of statistical mechanics, the most natural one is
to fix a very large planar domain and consider the uniform measure on \emph{all}
(finitely many) lozenge tilings of this domain. Asymptotic representation theory
suggests two more ways, related to decompositions of tensor products of
representations of classical Lie groups and restrictions of characters of the
infinite-dimensional counterparts of these groups, see e.g.\ \cite{BufGor},
\cite{BBO}.  Finally, one can also \emph{grow} tilings by means of interacting
particle systems, see e.g.\ \cite{BG}, \cite{BF}.

\smallskip

The rigorous asymptotic results available in the literature for various random
lozenge tilings models all share the same features. Let us describe such features on
the example of the uniformly random lozenge tilings of an $A\times B\times C$
hexagon as $A,B,C\to\infty$ (in such a way that $A/B$ and $A/C$ converge to
constants), see Figure \ref{Fig_Hex}. Note that such tilings can be linked both to
decompositions of certain irreducible representations (cf.\ \cite{BP}) and to growth
models (see \cite{BG}).

\begin{figure}[t]
\begin{center}
 {\scalebox{0.7}{\includegraphics{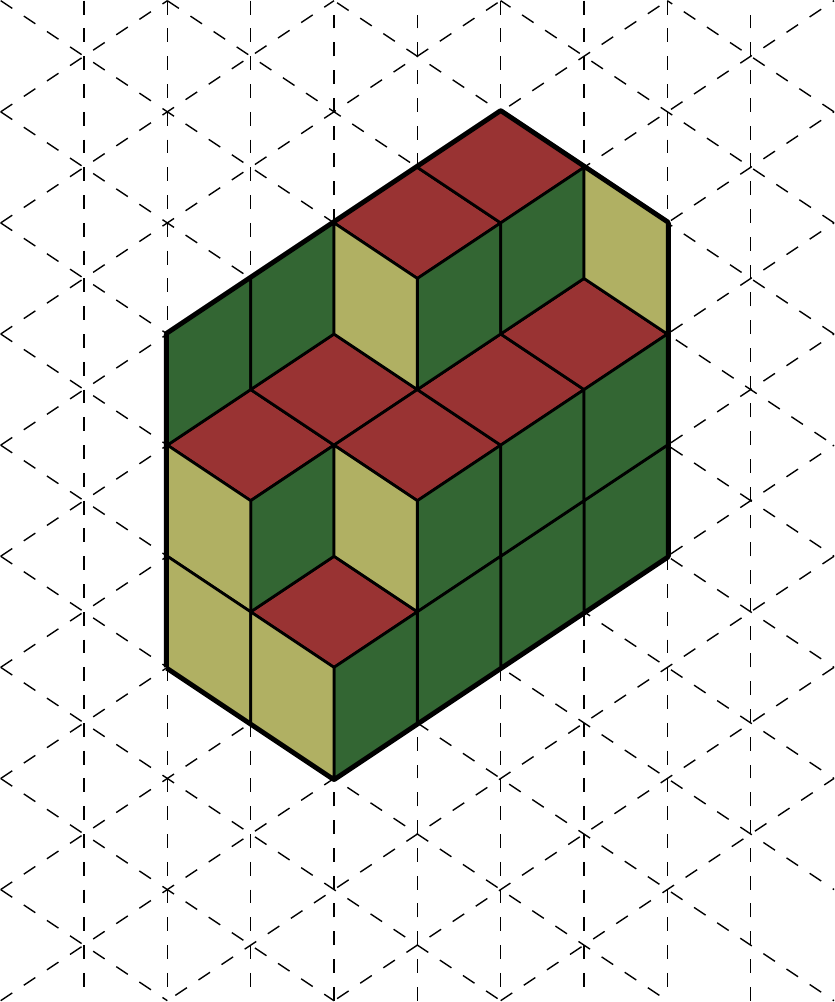}}}
 \quad
  {\scalebox{0.19}[0.31]{\includegraphics{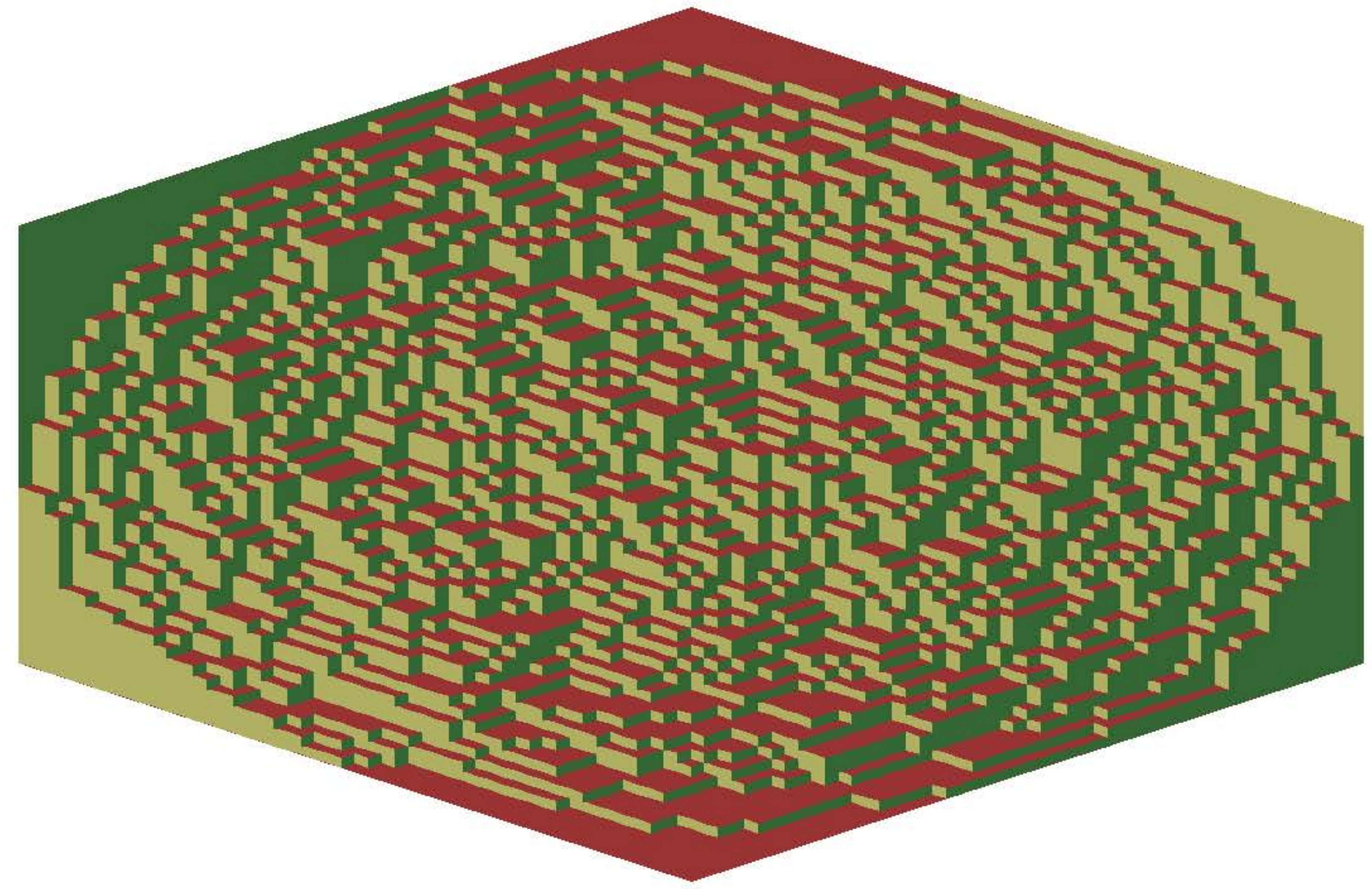}}}
\end{center}
 \caption{Uniformly random lozenge tilings of $3\times 4\times 2$ and $50\times 50\times 50$ hexagons.\label{Fig_Hex}}
\end{figure}

The following asymptotic features of the tilings of hexagons are known:
\begin{itemize}
 \item The random lozenge tilings exhibit the \emph{law of large numbers}, cf.\ Section \ref{Section_LLN}. One way
 to phrase it is that in each macroscopic sub-region of the hexagon the asymptotic proportions of three
 types of lozenges converge in probability to deterministic numbers, described by integrals over this sub-domain of three functions
 $p^{{\scalebox{0.15}{\includegraphics{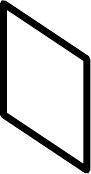}}}}(\cdot)$,
$p^{{\scalebox{0.15}{\includegraphics{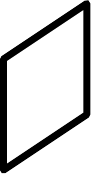}}}}(\cdot)$, and
$p^{{\scalebox{0.15}{\includegraphics{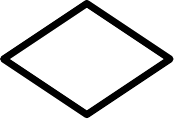}}}}(\cdot)$.  In particular,
 outside the inscribed ellipse one observes the \emph{frozen region} where asymptotically only one
 type of lozenges remains present, see \cite{CLP}, \cite{CKP}, \cite{BufGor}. The region
 where all three types of lozenges are asymptotically present is called
 \emph{liquid.}
 \item The global fluctuations of the \emph{height function} of tiling (see Section
 \ref{Section_LLN} for the detailed definition) are asymptotically normal and can be
 described via the pullback of the Gaussian Free Field with Dirichlet boundary
 conditions, see \cite{Petrov_GFF}, \cite{Duits}, \cite{BufGor_CLT}.
 \item Locally, near each point inside the inscribed ellipse in the limit one
 observes a \emph{translation invariant ergodic Gibbs measure} on lozenge tilings of the plane. Its\ \emph{slope} is determined by the law of large
 numbers (such measure is unique for each slope, see \cite{Sheffield}, and its
 correlation functions admit known closed expressions, cf.\ Section \ref{Section_Main_proof}), see \cite{BKMM},
 \cite{G_Hahn}, \cite{Petrov-curves}.
 \item The one--point fluctuations of the boundary of the frozen region at generic point after proper rescaling
 converge to the Tracy--Widom distribution $F_{GUE}$, and multi--point fluctuations
 are described by the Airy line ensemble, see \cite{BKMM}, \cite{Petrov-curves}.\footnote{For tilings of more complicated regions, the boundary of the
 frozen region might develop various singularities, which lead to different behaviors.}

 \item At a point where the frozen boundary is tangent to a side of the hexagon, the
 fluctuations are described by the GUE--corners process, see \cite{JN},
 \cite{Nord}, \cite{GP}, \cite{Novak}.
\end{itemize}
The \emph{universality belief} predicts that all these features should be very
robust along different models of lozenge tilings, and should not depend on details.
It means, that while the limit shape (i.e.\ the asymptotic proportions
$p^{{\scalebox{0.15}{\includegraphics{lozenge_v_down.pdf}}}}(\cdot)$,
$p^{{\scalebox{0.15}{\includegraphics{lozenge_v_up.pdf}}}}(\cdot)$,
$p^{{\scalebox{0.15}{\includegraphics{lozenge_hor.pdf}}}}(\cdot)$) and the exact
shape of the boundaries of the frozen regions differ from system to system, but they
should always exist. The fluctuations of the height function should always be given
by a pullback of the Gaussian Free Field, and only the map with respect to which
this pullback is taken, might change. Bulk local limits should be always described
by translation invariant Gibbs measures and only the slope of such measure might
vary from system to system. Finally, the fluctuations of the boundaries of frozen
regions should be always governed by the Tracy--Widom distribution, Airy line
ensemble and GUE--corners process.

While this conjectural universality is confirmed by numerous examples, but by now
only the Law of Large Numbers has been proven in sufficient generality, cf.\
\cite{CKP}, \cite{KO}, \cite{BufGor}.

\bigskip

In the present article we address the third feature about local (also called ``bulk'') scaling
limits. In simple words, we show that whenever for a random lozenge tiling the law of large numbers
holds, if the tiled domain has a straight boundary, then near this boundary the \emph{universality}
for the local limits is valid.

\subsection{Results}

We proceed to a more detailed formulation of our main result.

Consider a trapezoid drawn on the triangular grid as shown in Figure \ref{Fig_trap}.
Such domain is parameterized by the length of the left vertical side $C$ and the
width $N$ and we denote it $\Omega_{N,C}$. Consider the set of all tilings of
$\Omega_{N,C}$, in which we allow horizontal lozenges to \emph{stick out} of the
right boundary. The combinatorial constraints imply that there are precisely $N$
horizontal lozenges sticking out. Let $\mathcal P$ be a probability measure on
tilings of $\Omega_{N,C}$ which satisfies the \emph{Gibbs property}, which means
that given the positions of the $N$ horizontal lozenges on the boundary the
conditional distribution of the tilings becomes uniform.

\begin{figure}[t]
\begin{center}
 {\scalebox{0.6}{\includegraphics{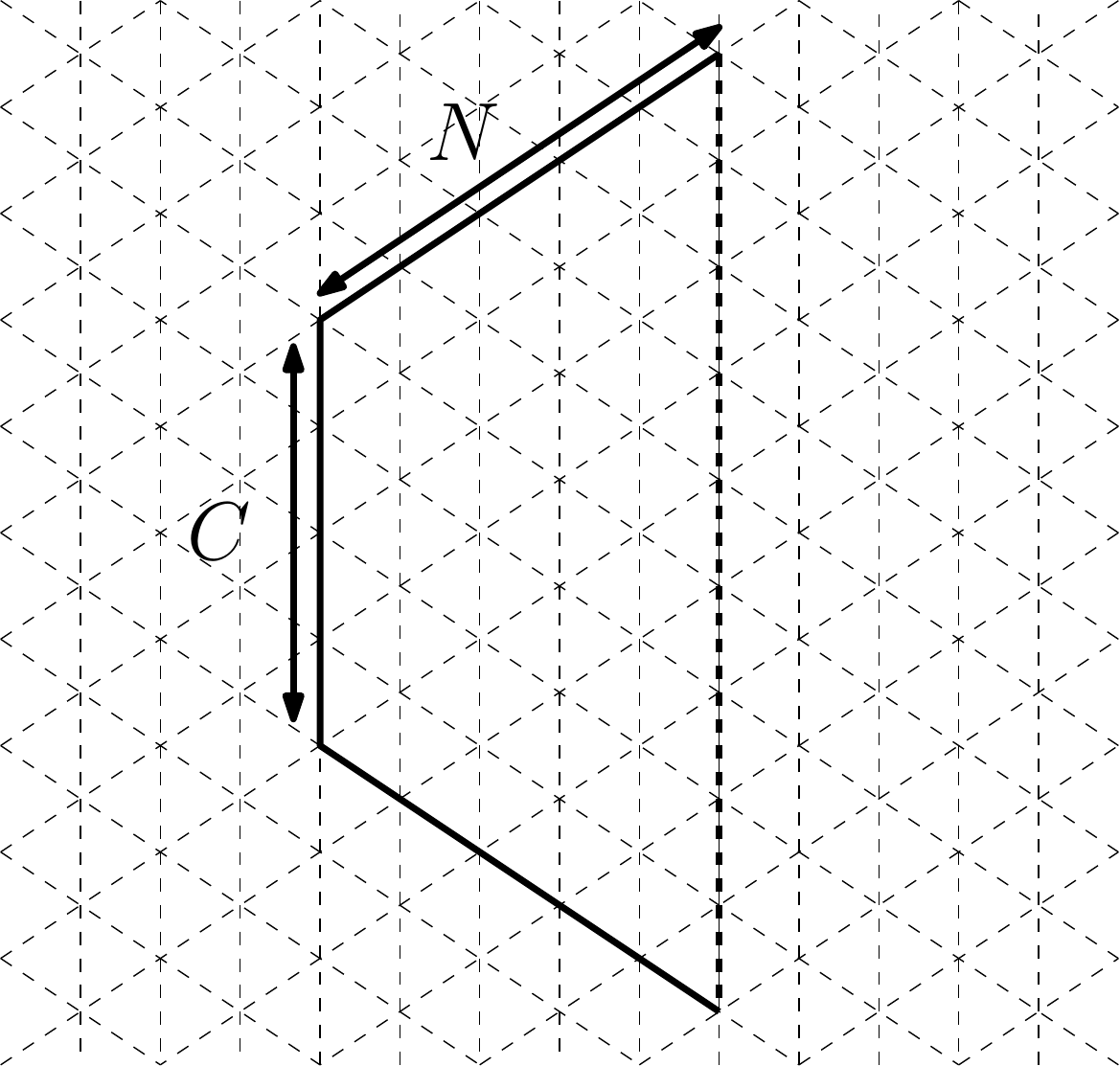}}}
 \quad
  {\scalebox{0.6}{\includegraphics{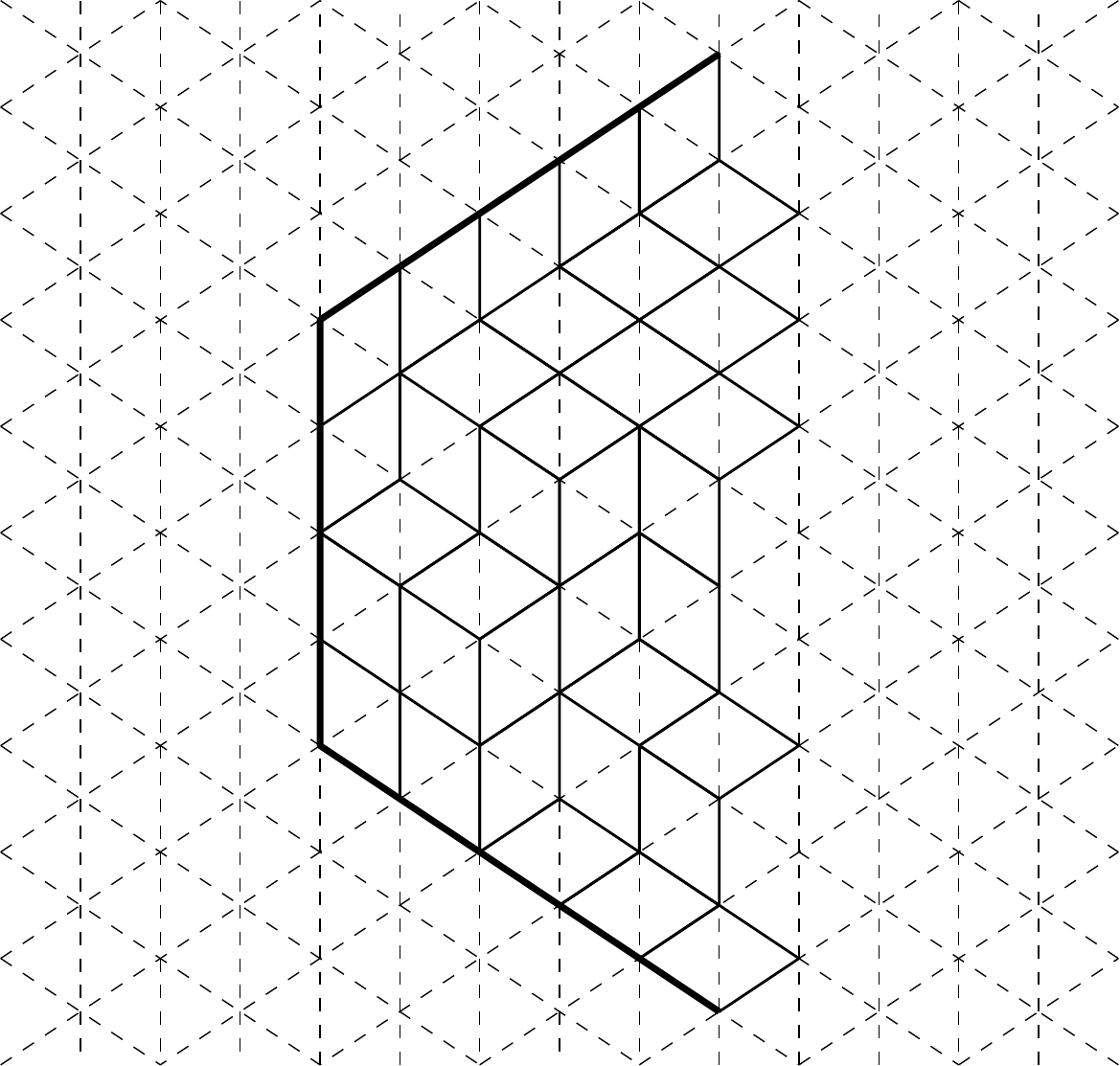}}}
\end{center}
 \caption{Trapezoid $\Omega_{N,C}$ with $N=5$, $C=4$ and one of its tilings.\label{Fig_trap}}
\end{figure}

Further, given a tiling $\omega$ of $\Omega_{N,C}$ with coordinates of the sticking
out horizontal lozenges $\ell_1>\ell_2>\dots>\ell_N$, let
$$
\mu(\omega)=\frac{1}{N}\sum_{i=1}^N \delta_{\ell_i/N}
$$
be a probability measure encoding them.
 Then the pushforward of $\mathcal P$ with
respect to the map $\omega\mapsto \mu(\omega)$ is a \emph{random} probability
measure $\mu_{\mathcal P}$. Due to the Gibbs property, $\mathcal P$ is uniquely
reconstructed by $\mu_{\mathcal P}$.

Note that each tiling of $\Omega_{N,C}$ can be viewed also as a tiling of
$\Omega_{N,C+1}$, by adding a row of
{{\scalebox{0.2}{\includegraphics{lozenge_v_up.pdf}}}} lozenges on top. For example,
the tiling $\Omega_{5,4}$ in the right panel of Figure \ref{Fig_trap} is
simultaneously a tiling of $\Omega_{5,3}$ with added row on top. Because of this
correspondence, if we fix the coordinates of horizontal lozenges along the right
boundary, then the exact value of $C$ in $\Omega_{N,C}$ becomes not important (as
long as it is large enough). In the same way, in order to reconstruct a $\mathcal
P$--random tiling by $\mu_{\mathcal P}$ we actually do not need to know the value of
$C$.

\begin{theorem} \label{Theorem_main_intro}
 Let $\Omega_{N,C(N)}$, $N=1,2,\dots$, be a sequence of trapezoids equipped
 with probability measures $\mathcal P(N)$ on their lozenge tilings. Suppose that:
 \begin{itemize}
  \item[$(A)$] $\frac{C(N)}{N}$ stays bounded as $N\to\infty$.
  \item[$(B)$] The random probability measures $\mu_{\mathcal P(N)}$ converge (weakly, in
  probability) to a \emph{deterministic} probability measure $\mu$.
 \end{itemize}
 Then
 \begin{enumerate}
  \item The $\mathcal P(N)$--random lozenge tilings of $\Omega_{N,C(N)}$ exhibit the
  Law of Large Numbers as $N\to\infty$: there are three deterministic densities $p^{{\scalebox{0.15}{\includegraphics{lozenge_v_down.pdf}}}}(\mathbf{x},\boldsymbol{\eta})$,
$p^{{\scalebox{0.15}{\includegraphics{lozenge_v_up.pdf}}}}(\mathbf{x},\boldsymbol{\eta})$,
and
$p^{{\scalebox{0.15}{\includegraphics{lozenge_hor.pdf}}}}(\mathbf{x},\boldsymbol{\eta})$,
$0<\boldsymbol{\eta}<1$, $\mathbf x\in\mathbb R$, such that for any subdomain
$\mathcal D\subset \frac{1}{N} \Omega_{N,C(N)}$ with smooth boundary, the normalized
by $N$ (random) numbers of lozenges of types
$({\scalebox{0.2}{\includegraphics{lozenge_v_down.pdf}}},
{\scalebox{0.2}{\includegraphics{lozenge_v_up.pdf}}},
{\scalebox{0.2}{\includegraphics{lozenge_hor.pdf}}})$ inside $N \mathcal D$ converge
in probability to the vector
$$
 \left( \int_{\mathcal D} p^{{\scalebox{0.15}{\includegraphics{lozenge_v_down.pdf}}}}(\mathbf{x},\boldsymbol{\eta}) d\mathbf x d\boldsymbol{\eta},
\int_{\mathcal D}
p^{{\scalebox{0.15}{\includegraphics{lozenge_v_up.pdf}}}}(\mathbf{x},\boldsymbol{\eta})
d\mathbf x d\boldsymbol{\eta}, \int_{\mathcal
D}p^{{\scalebox{0.15}{\includegraphics{lozenge_hor.pdf}}}}(\mathbf{x},\boldsymbol{\eta})
d\mathbf x d\boldsymbol{\eta} \right)
$$
 \item Take a point $(\mathbf x, \boldsymbol{\eta})$, such that all three asymptotic
 densities at this point are continuous and non-zero. Assume $0<\boldsymbol{\eta}<1$. If $n(N)$, $x(N)$ are two sequences of
 integers, such that $\lim_{N\to\infty} n(N)/N=\boldsymbol{\eta}$, $\lim_{N\to\infty}
 x(n)/N=\mathbf x$, then the point process of lozenges near the point $(n(N), x(N))$
 weakly converges as $N\to\infty$ to the (unique) translation invariant ergodic Gibbs
 measure on lozenge tilings
 of slope $(p^{{\scalebox{0.15}{\includegraphics{lozenge_v_down.pdf}}}}(\mathbf{x}, \boldsymbol{\eta}),
p^{{\scalebox{0.15}{\includegraphics{lozenge_v_up.pdf}}}}(\mathbf{x},
\boldsymbol{\eta}),
p^{{\scalebox{0.15}{\includegraphics{lozenge_hor.pdf}}}}(\mathbf{x},
\boldsymbol{\eta}))$.
 \end{enumerate}
\end{theorem}
A more detailed formulation and the proof is below in Theorems \ref{Theorem_main_bulk},
\ref{Theorem_LLN_bulk}. Of course, for the domains obtained from $\Omega_{N,C}$ by rotating it by
$60\cdot k$ degrees, $k=1,2,3,4,5$, the exact analogue of Theorem \ref{Theorem_main_intro} is also
valid.

\bigskip

The conditions $(A)$ and $(B)$ of Theorem \ref{Theorem_main_intro} are known to hold
for many models of lozenge tilings, which implies that the local convergence to the
translation invariant Gibbs measures is true for them. Let us provide some examples:

\begin{figure}[t]
\begin{center}
 {\scalebox{0.6}{\includegraphics{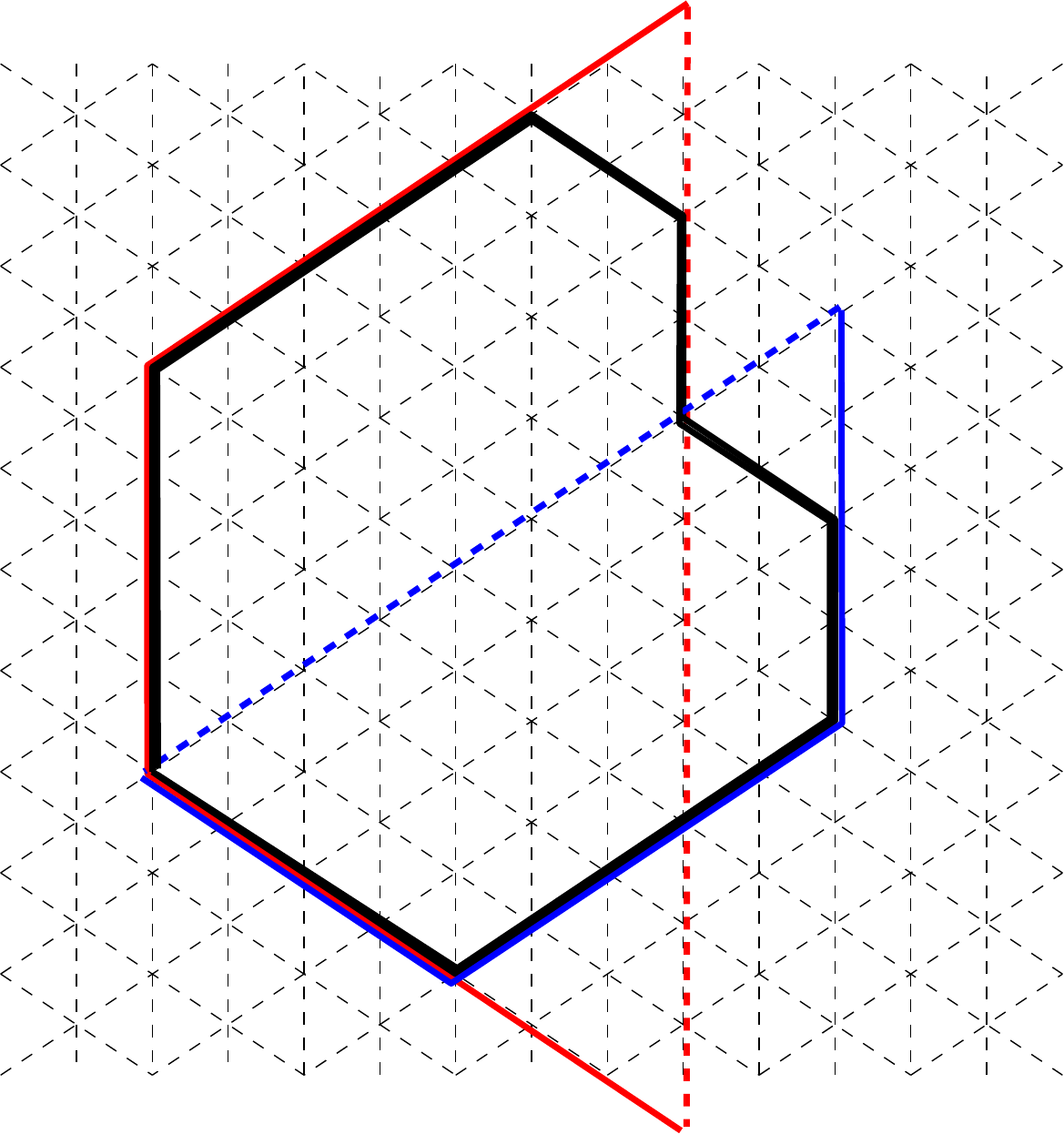}}}
 \quad
  {\scalebox{0.6}{\includegraphics{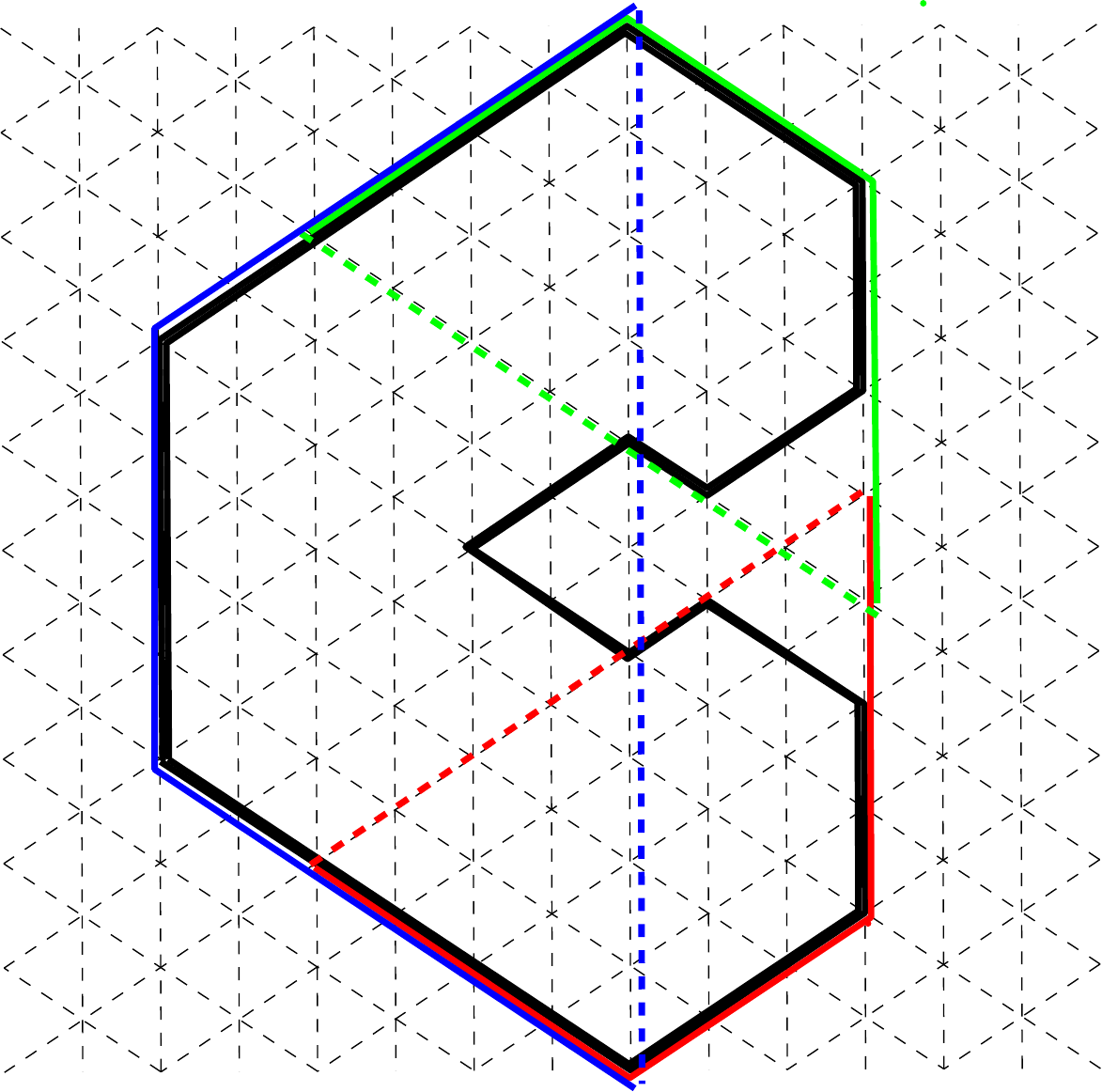}}}
\end{center}
 \caption{Left panel: the heart--shaped polygon (in black) is covered by two trapezoids. Right panel: the C--shaped polygon (in black)
 is covered by three trapezoids. \label{Fig_covered_domains}}
\end{figure}

\begin{corollary} \label{Corollary_tilings}
 Let $\Omega$ be any simply-connected polygonal domain on the triangular grid, and
 define $\Omega(L)$, $L=1,2,\dots$ to be the domain obtained by multiplying all the
 side lengths of $\Omega$ by $L$. For any part of $\Omega(L)$ covered by a trapezoid, near any
 point in the liquid region in this part, the uniformly random lozenge tilings of
 $\Omega(L)$ converge locally to the ergodic translation--invariant Gibbs measure of
 the corresponding slope.
\end{corollary}

Let us elaborate on the notion of a part of $\Omega(L)$ covered by a trapezoid. Of
course, one can always consider a huge trapezoid, such that the the entire domain
$\Omega(L)$ will be inside. However, that's not what we want. We rather require the
part to be such that the restrictions of the uniformly random tilings of the entire
domain to this part are described by Gibbs probability measures on tilings of the
trapezoid in the context of Theorem \ref{Theorem_main_intro}. In particular, three
 sides of such trapezoid should belong to the same lines as sides of $\Omega(L)$,
 cf.\ Figure \ref{Fig_covered_domains}. Section \ref{Section_Main_proof} explains
 this notion in more details and culminates in Theorem \ref{Theorem_tilings} which
 is a refinement of Corollary \ref{Corollary_tilings}.

\begin{remark}
 Many polygonal domains can be \emph{completely covered} by trapezoids,
 which implies the convergence to the ergodic translation--invariant Gibbs measures
 \emph{everywhere} in the liquid region, cf.\ Figure \ref{Fig_covered_domains} for
 examples. Yet more complicated domains are covered only partially, cf.\ Figure
 \ref{Fig_not_covered_domain}.
\end{remark}
\begin{remark}
 For specific polygons, which are covered by a \emph{single} trapezoid, an analogue of Corollary
 \ref{Corollary_tilings} was previously proven in \cite{Petrov-curves}. For a class of
 (non-polygonal) domains, such that the limit shape has \emph{no frozen regions} an analogue of
 Corollary \ref{Corollary_tilings} was previously proven in \cite{Kenyon_GFF}. A general conjecture
 that the bulk asymptotic behavior of Corollary \ref{Corollary_tilings} should hold
 \emph{universally} for tilings of finite planar domains dates back to \cite{CKP}.
\end{remark}
\begin{remark}
 The tiled domains do not have to be polygonal, see Theorem \ref{Theorem_tilings}
 for the detailed formulation. The assumption of being simply--connected is also
 probably not essential, however, most Law of Large Numbers type theorems in the
 literature stick to this assumption for simplicity, and so we have to use it here
 as well. There are examples of domains with holes for which the Law of Large Numbers is explicitly
 known (see e.g.\ \cite[Section 9.2]{BGG}), and for them Corollary \ref{Corollary_tilings} holds.
\end{remark}

\begin{figure}[t]
\begin{center}
 {\scalebox{0.6}{\includegraphics{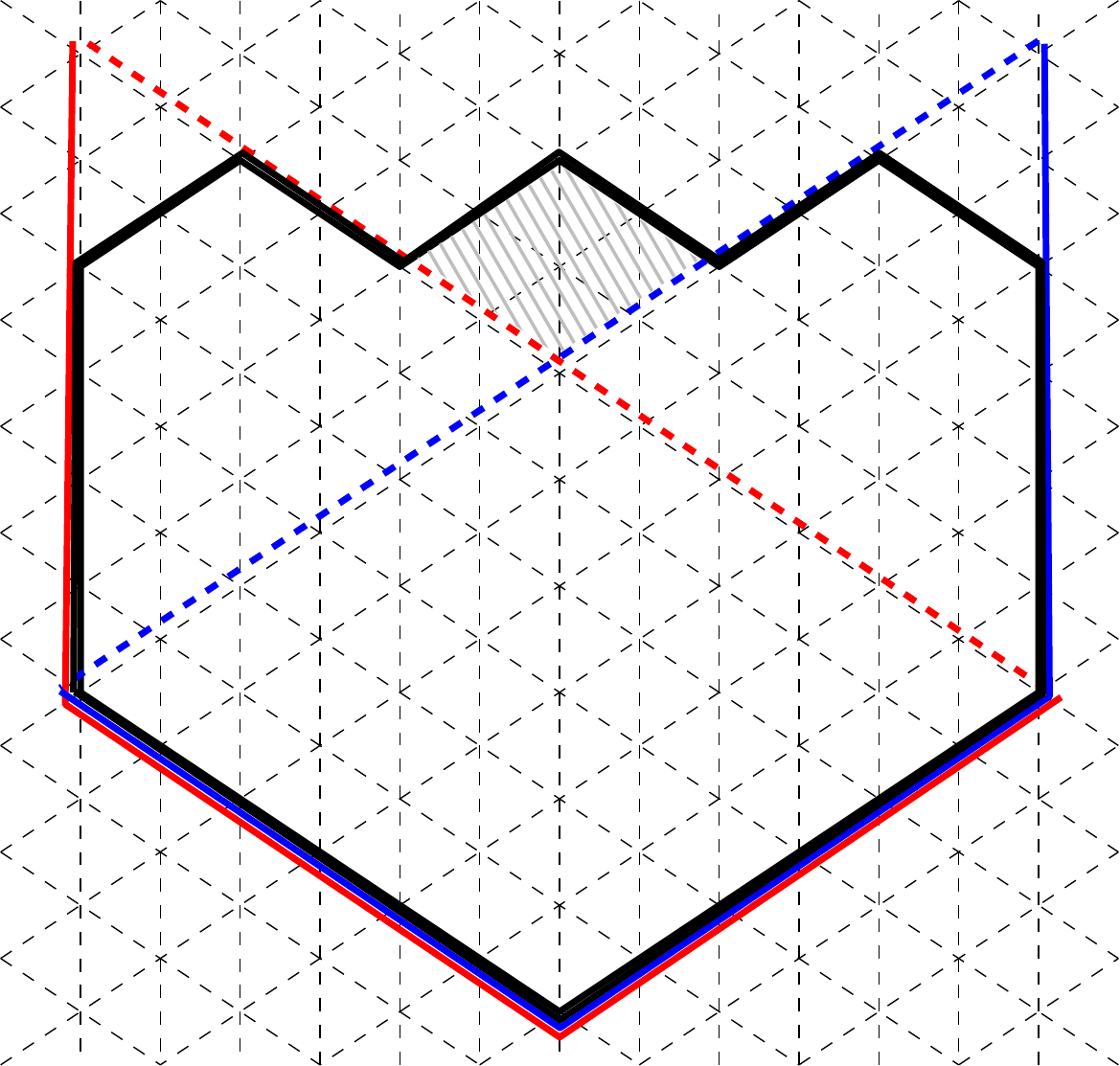}}}
\end{center}
 \caption{Most of this domain is covered by two trapezoids, yet we can not cover the shaded gray
 region. Therefore, Corollary \ref{Corollary_tilings} applies everywhere except the gray
 region.
  \label{Fig_not_covered_domain}}
\end{figure}

Another example comes from the representation theory. Recall that the irreducible representations
of the $N$--dimensional unitary group $U(N)$ are parameterized by $N$--tuples of integers
$\lambda_1\ge \lambda_2\ge \dots\ge\lambda_N$ called \emph{signatures}, see e.g.\ \cite{Weyl}. It
is convenient to shift the coordinates by introducing strictly ordered coordinates
$\ell_i=\lambda_i-i$, $i=1,\dots,N$. We denote the corresponding irreducible representation through
$T_\ell$. Such representation has a distinguished \emph{Gelfand--Tsetlin basis}, parameterized by
the \emph{Gelfand--Tsetlin patterns}, which are in bijection with lozenge tilings of trapezoidal
domains as above. Here the coordinates of the horizontal lozenges sticking out of the right
boundary of the domain are precisely the label $\ell_1>\ell_2>\dots>\ell_N$ of the representation,
see \cite{BP} and also Section \ref{Section_GT_assy} for the details. In particular, the dimension
$\dim(T_\ell)$ equals the total number of the tilings of trapezoidal domain with fixed $N$ lozenges
along the right boundary.

If we take any \emph{reducible} (finite--dimensional) representation $T$ of $U(N)$, then we can
decompose it into irreducible components $T=\oplus_{\ell} c_\ell  T_\ell$ and further consider the
Gelfand--Tsetlin basis in each of them. Recalling the bijection with tilings, we thus determine for
each lozenge tiling of trapezoidal domain a non-negative number, which is equal to $c_\ell$ with
$\ell$ encoding the horizontal lozenges along the right vertical boundary of the domain. All these
numbers sum up to $\dim(T)$, so that after dividing by $\dim(T)$ they define a probability measure
$\mathcal P^T$ on lozenge tilings. Varying $T$ in this construction one arrives at several
intriguing probability distributions.

The first example is given by the tensor product of two irreducible representations.
\begin{corollary} \label{Corollary_Tensor_products}
 Let $\ell(N)$, $\kappa(N)$ be two sequences of signatures (in the notation with strictly increasing coordinates) such that
 \begin{itemize}
 \item  The numbers $\frac{\ell_i(N)}{N}$, $\frac{\kappa_i(N)}{N}$, $1\le i \le N$ are uniformly
  bounded.
  \item There exist two monotonous functions $f^{\ell}$, $f^{\kappa}$ with finitely many points of discontinuity and such that
 $$
  \lim_{N\to\infty} \frac{1}{N}\sum_{i=1}^N \left|
  f^{\ell}\left(\frac{i}{N}\right)-\frac{\ell_i(N)}{N}\right|=   \lim_{N\to\infty} \frac{1}{N}\sum_{i=1}^N \left|
  f^{\kappa}\left(\frac{i}{N}\right)-\frac{\kappa_i(N)}{N}\right|.
 $$
 \end{itemize}
Consider the \emph{tensor product} $T(N)=T_{\ell(N)}\otimes T_{\kappa(N)}$ and the corresponding
measure on lozenge tilings $\mathcal P^{T(N)}$. Then the conclusion of Theorem
\ref{Theorem_main_intro} is valid for  $\mathcal P^{T(N)}$ as $N\to\infty$.
\end{corollary}
The proof of Corollary \ref{Corollary_Tensor_products} is a combination of Theorem
\ref{Theorem_main_intro} with \cite[Theorem 1.1]{BufGor} describing the Law of Large Numbers for
the tensor products.

We remark that since Theorem \ref{Theorem_main_intro} is restricted to
$0<\boldsymbol{\eta}<1$, we do not get any information about $\boldsymbol{\eta}=1$,
i.e.\ the behavior of the lozenges at the right boundary of the domain.
\emph{Conjecturally} this behavior should be similar, cf.\ \cite{BES} for the recent
results in this direction for sums of random matrices, which are continuous
analogues of tensor products (see e.g.\ \cite[Section 1.3]{BufGor}).

\medskip

Another celebrated representation of $U(N)$ is $(\mathbb C^N)^{\otimes n}$, which was intensively
studied in the context of the \emph{Schur--Weyl duality}, cf. \cite{Weyl}.

\begin{corollary} \label{Corollary_Schur-Weyl}
Consider the representation of $U(N)$ in $(\mathbb C^N)^{\otimes n}$ by natural action in each
component of the tensor product. Suppose that as $N\to\infty$, $n$ varies in such a way that
$\lim_{N\to\infty} n/N^2=c>0$. Let $\mathcal P^N$ denote the probability measure on lozenge tilings
corresponding to the decomposition of this representation. Then the conclusion of Theorem
\ref{Theorem_main_intro} is valid for  $\mathcal P^{N}$ as $N\to\infty$.
\end{corollary}
The proof of Corollary \ref{Corollary_Schur-Weyl} is a combination of Theorem
\ref{Theorem_main_intro} with the Law of Large Numbers for the decomposition of $(\mathbb
C^N)^{\otimes n}$ obtained in \cite{Biane}, see also \cite[Theorem 5.1]{BufGor}

\medskip

Corollaries \ref{Corollary_tilings}, \ref{Corollary_Tensor_products},
\ref{Corollary_Schur-Weyl} do no exhaust the list of possible applications of
Theorem \ref{Theorem_main_intro} and we refer e.g.\ to \cite{BBO}, \cite{Panova_1},
\cite{Panova_2} for some further examples of the situations where this theorem
holds. Further, Theorem \ref{Theorem_main_intro} also has consequences for
\emph{domino} tilings on square grid, as in \cite{BufKnizel}. In more details, it
implies that in the particle process corresponding to the rectangular parts of
domains tiled with dominos, the 1d bulk scaling limit along a section is universally
governed by the discrete Sine process, see Section \ref{Section_Main_proof} for the
definition of the discrete Sine process and \cite[Appendix B]{BufKnizel} for the
exact statement.

\subsection{Discussion}

One way to interpret Theorem \ref{Theorem_main_intro} is that straight boundaries
lead to a \emph{smoothing effect}. Indeed, we do not know much about the local
structure of random measures $\mu_{\mathcal P(N)}$, and it can be anything, yet as
soon as we move macroscopic distance towards the straight boundary, the local
measures become universal. This interpretation can be put into a wider context by
observing that the stochastic process of $N-t$ horizontal lozenges on $(N-t)$th
(from the right) vertical line of a trapezoid $\Omega_{N,C}$  is a Markov chain in
time variable $t$, as follows from the Gibbs property. Therefore, we see that this
Markov chain has a \emph{homogenization} property, i.e.\ its local statistics become
universal. In a parallel work \cite{GorPet} we study a similar homogenization for
the families of non-intersecting paths on very short time scales.

In the continuous setting there is a close analogy with the homogenization
properties of \emph{Dyson Brownian Motion}. Such property was first observed in
\cite{J_Wigner}, and used there for proving the universality of the local statistics
for certain ensembles of Wigner random matrices. More recently, such homogenization
was developed much further and has led to many exciting universality results, see
e.g.\ \cite{Shch}, \cite{BEY}, \cite{LY} and references therein.

\smallskip

It is natural to ask whether an analogue of Theorem \ref{Theorem_main_intro} can
hold for other conjectural universal features of the lozenge tilings, such as the
fluctuations of the frozen boundary or global fluctuations of the height function.
Regrettably, the answer is \emph{no}, the knowledge of the Law of Large Numbers at
the right boundary of the trapezoid domain is not enough for proving other universal
features. A simple way to see that is to take a Bernoully random variable
$\varsigma$ and to add $\lfloor \varsigma \sqrt{N} \rfloor$ to the coordinates of
the lozenges at the right boundary. This addition clearly does not change the Law of
Large Numbers (the assumption of Theorem \ref{Theorem_main_intro}), yet it will lead
to the same shift everywhere in the tiling, and therefore will influence all the
conjectural limiting behaviors, \emph{except} for the local bulk limits that we
study here. We also refer to \cite[Section 1.9]{DJM} for a related discussion.

\subsection{Our methods}
 The proof of Theorem \ref{Theorem_main_intro} is based on several ingredients. The
 first one is the double contour integral expression of \cite{Petrov-curves} for the
 correlation kernel of the determinantal point process describing the uniformly
 random lozenge tilings of trapezoids with \emph{fixed} positions of
 horizontal lozenges on the right boundary. Our asymptotic analysis of this kernel
 reveals a new fact:  the bulk
asymptotic behavior depends only on the global scale law of large numbers for the
right boundary. This paves a way to pass from fixed to random horizontal lozenges on
the right boundary. Let us emphasize that in the latter random boundary case the
correlation functions of the point process of lozenges inside the domain do not have
to be determinantal (and we do not expect any explicit formulas for them); yet, this
turns out to be irrelevant for our asymptotic analysis.

 We also
need to make a link between the slope of the local limiting Gibbs measure and the slope in the
global Law of Large Numbers, for which we make use of the description of \cite{BufGor} of the limit
shape through the quantization of the Voiculescu $R$--transform.

\subsection{Acknowledgements}

The author would like to thank Alexei Borodin, Alexey Bufetov and Leonid Petrov for
helpful discussion. This work was partially supported by the NSF grant DMS-1407562
and by the Sloan Research Fellowship.

\section{Gelfand--Tsetlin patterns with fixed top row}

\label{Section_GT_assy}

For a fixed $N=1,2,\dots$ let $\t=(\t_1>\t_2>\dots>\t_N)$ be an $N$ tuple of integers. A
\emph{Gelfand--Tsetlin pattern} with top row $\t$ is an array of $N(N+1)/2$ integers $x^j_i$, $1\le
i \le j \le N$, which satisfy the interlacing condition
$$
  x^{i+1}_j \ge x^i_j > x^{i+1}_{j+1},\quad 1 \le i \le j \le N-1,
$$
and the top row condition $x^N_i=\t_i$, $i=1,\dots,N$. Such patterns are in bijection with lozenge
tilings of certain specific domains, shown in Figure \ref{Fig_GT_tiling}; the bijection is given by
positions of horizontal lozenges.\footnote{Note that in many articles an alternative
parameterization $y^i_j=x^i_j+j$ is used. We stick to the notations of \cite{Petrov-curves}, as we
use many ideas from that paper.}

\begin{figure}[h]
\begin{center}
 {\scalebox{0.9}{\includegraphics{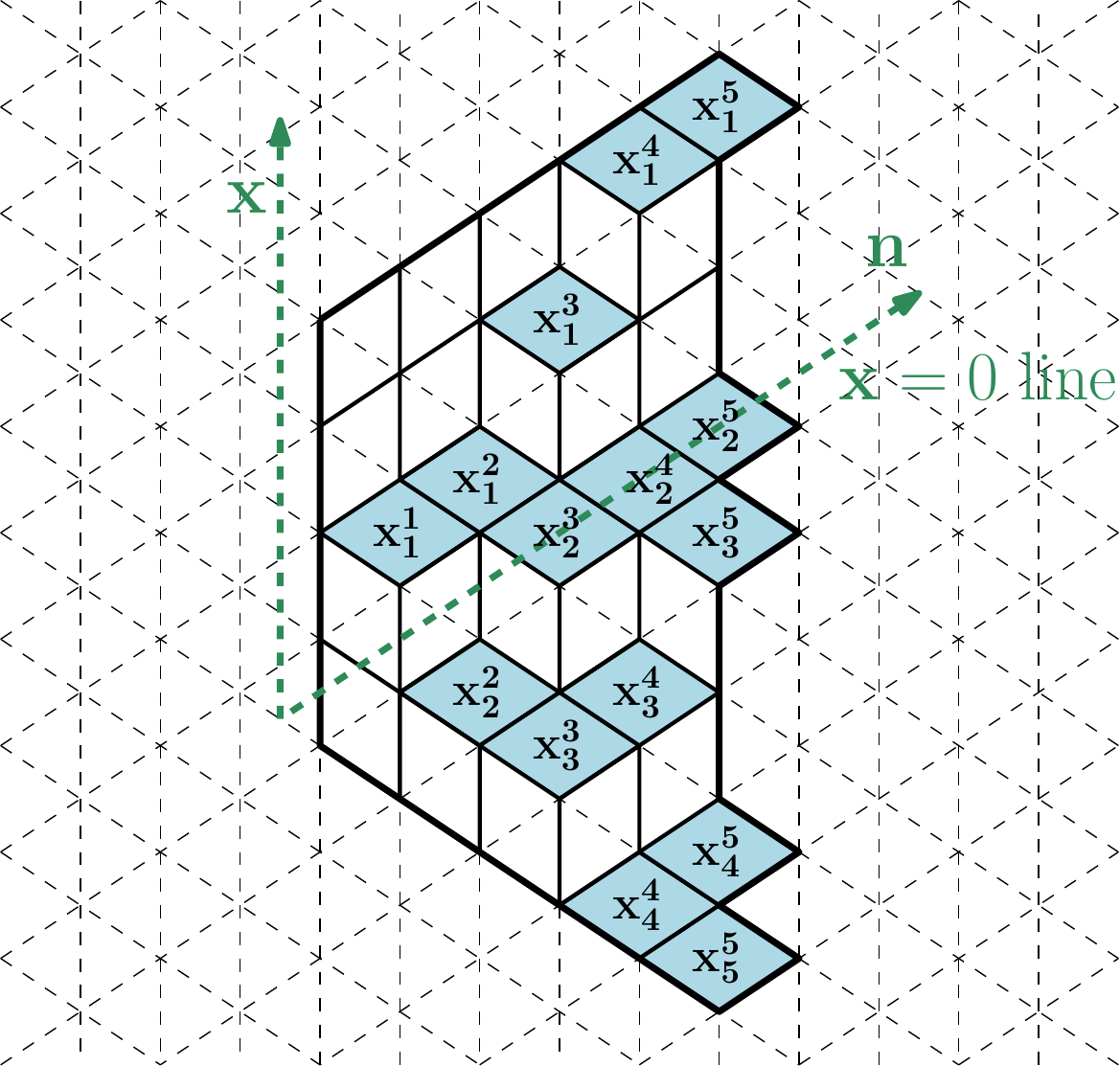}}}
\end{center}
 \caption{A lozenge tiling corresponding to a Gelfand--Tsetlin pattern with $N=5$ and
 top row $x^5=\t=(3,0,-1,-4,-5)$.\label{Fig_GT_tiling}}
\end{figure}

In this section we investigate the asymptotic behavior of \emph{uniformly random}
Gelfand--Tsetlin patterns with fixed top row as $N\to\infty$. We adapt important
ideas from \cite{Petrov-curves}; yet the results of that article are not enough for
our purposes, so we need to generalize them and supplement with new considerations.

\bigskip

The following theorem of \cite{Petrov-curves} is our starting point, see also \cite{DM}.

\begin{theorem}[{\cite[Theorem 5.1]{Petrov-curves}}] \label{Theorem_Leo}
 Fix an $N$--tuple of integers $\t=(\t_1>\t_2>\dots\t_N)$, and let $\{x_i^j\}$, $1\le i \le j \le N$ be
 uniformly random Gelfand--Tsetlin pattern with top row $\t$. Then for any $k$ and any collection of distinct pairs of
 integers $(x(1),n(1)),\dots,(x(k),n(k))$ with $1\le n(i)<N$, $i=1,\dots,k$, we have
 $$
  {\rm Prob } \left[ x(i)\in \{x_1^{n(i)},x_2^{n(i)},\dots,x_j^{n(i)}\},\,
  i=1,\dots,k\right]=\det_{i,j=1}^k \left[ K(x(i),n(i);x(j),n(j))\right],
 $$
 where
 \begin{multline} \label{eq_correlation_kernel}
 K(x_1,n_1;x_2,n_2)= -\1_{n_2<n_1} \1_{x_2\le x_1}
 \frac{(x_1-x_2+1)_{n_1-n_2-1}}{(n_1-n_2-1)!} + \frac{(N-n_1)!}{(N-n_2-1)!} \\
 \times \frac{1}{(2\pi \ii)^2} \oint_{\C(x_2,\dots,\t_1-1)} dz \oint_{\C(\infty)} dw
 \frac{(z-x_2+1)_{N-n_2-1}}{(w-x_1)_{N-n_1+1}} \, \frac{1}{w-z} \prod_{r=1}^N
 \frac{w-\t_r}{z-\t_r},
\end{multline}
the contour $\C(x_2,\dots,\t_1-1)$ encloses points $x_2,x_2+1,\dots,\t_1-1$ and no other
singularities of the integrand, and the contour $\C(\infty)$ is a very large circle; both contours
have positive orientation.
\end{theorem}

The content of this section is the asymptotic analysis of the correlation kernel
\eqref{eq_correlation_kernel}. Although we use a somewhat standard steepest descent
approach to such analysis (cf.\ \cite{Okounkov}, \cite{BG_Lectures},
\cite{Petrov-curves}), but the technical details are delicate, as we need to deal
with arbitrary $N$--tuples $\t$.

\bigskip

We start by rewriting the integrand in the double integral of \eqref{eq_correlation_kernel} as
$$
\frac{1}{w-z} \cdot \frac{1}{(w-x_1)(z-x_2+N-n_2)} \cdot \exp(G_2(z)-G_1(w))
$$
with
\begin{equation}
\label{eq_action_logarithm} G_\kappa(z)=\sum_{a=1}^{N-n_\kappa} \ln(z-x_\kappa+a) - \sum_{r=1}^N
\ln(z-\t_r), \quad \kappa=1,2.
\end{equation}
 Let us analyze the zeros
of the derivative of $G_\kappa(z)$, i.e.\ the solutions to
\begin{equation}
\label{eq_critical_equation}
 \sum_{a=1}^{N-n_\kappa} \frac{1}{z-x_\kappa+a}= \sum_{r=1}^N \frac{1}{z-\t_r}
\end{equation}
\begin{lemma} \label{Lemma_number_of_roots}
 The equation \eqref{eq_critical_equation} has either $0$ or $2$ non-real roots. In
 the latter case the non-real roots are complex conjugate to each other.
\end{lemma}
\begin{proof}
Set $\T=\{\t_1,\t_2,\dots, \t_N\}$ and $\Tc=\mathbb Z\setminus \T$. Then
\eqref{eq_critical_equation} can be written in the form
\begin{equation}
\label{eq_critical_rewritten}
 \sum_{d\in \T \setminus [x_\kappa-N+ n_\kappa, x_\kappa-1] } \frac{1}{z-d} - \sum_{d'\in \Tc \bigcap [x_\kappa-N+ n_\kappa, x_\kappa-1]}
 \frac{1}{z-d'}=0
\end{equation}
Let $M$ denote the total number of terms in \eqref{eq_critical_rewritten}. Then
after clearing the denominators, the equation \eqref{eq_critical_rewritten} becomes
a polynomial equation of degree $M-1$. Since all the coefficients of this equation
are real, all its non-real roots split into complex--conjugate pairs. We will now
show that \eqref{eq_critical_rewritten} has at least $M-3$ real roots, which then
implies that there is at most one complex conjugate pair.

Let $d^{l}_1<\dots<d^{l}_L$ be those elements of $\T$ which are smaller than $x_\kappa-N+n_\kappa$
and let $d^r_1<\dots<d^r_R$ be those which are greater than $x_\kappa-1$. Also let
$d'_1<\dots<d'_P$ be elements of $\Tc \bigcap [x_\kappa-N+ n_\kappa, x_\kappa-1]$. Clearly,
$L+R+P=M$. Further, for each $i=1,2,\dots,L-1$ the function \eqref{eq_critical_rewritten}
continuously changes from $+\infty$ to $-\infty$ on the real interval $(d^l_i,d^l_{i+1})$ and
therefore has a zero on this interval. Similarly, there is a zero on each interval
$(d^r_i,d^r_{i+1})$, $i=1,\dots,R-1$ and on each interval $(d'_i,d'_{i+1})$, $i=1,\dots,P-1$.
Summing up, we found $(L-1)+(R-1)+(P-1)=M-3$ distinct real roots of \eqref{eq_critical_rewritten}.
\end{proof}

We will further stick to the case when \eqref{eq_critical_rewritten} has a pair of non-real roots
(the asymptotic analysis in the case when all roots are real is more delicate and is left out of
the scope of the present paper) and denote these roots $\tau_\kappa$ and $\bar \tau_\kappa$ with
convention $\Im(\tau_\kappa)>0$. Let us state the main result of this section.

\begin{theorem} \label{Theorem_Kernel_asympt}
 Fix an arbitrary real parameter $D>0$. Assume that the pairs $(x_1,n_1)$ and $(x_2,n_2)$ are such that the
 complex critical points $\tau_1$ and $\tau_2$ of $G_1(z)$ and $G_2(z)$,
 respectively, satisfy $\Im(\tau_{1,2})>D^{-1} N $ and $|\tau_{1,2}|<DN$. Further, assume $|x_1-x_2|+|n_1-n_2|<D$ and
 that the points $\t_1,\dots,\t_N$ satisfy $\sum_{i=1}^N \ln^{1+1/D}(1+|\t_i|/N)<DN$. Then
 the kernel \eqref{eq_correlation_kernel} satisfies
\begin{multline} \label{eq_Kernel_bulk_asymptotic}
 K(x_1,n_1;x_2,n_2)\\=\frac{(1-n_1/N)^{1+n_2-n_1}}{2\pi \ii} \int_{ N^{-1}\bar \tau_1}^{N^{-1} \tau_1}
 \frac{(u-x_1/N)^{x_2-x_1-1}}{ (u-x_1/N+1-n_1/N)^{x_2-x_1+n_2-n_1+1}} du + o(1),
\end{multline}
where $o(1)$ is the \emph{uniformly small} remainder as $N\to\infty$, and the
integration contour crosses the real axis to the right from $x_1/N$ when $n_2\ge
n_1$, and on the interval $(\frac{x_1+n_1-1}{N}, \frac{x_1}{N})$ when $n_2<n_1$.
\end{theorem}
\begin{remark}
 The change of variables $w=\frac{u-x_1/N}{u-x_1/N+1-n_1/N}$, i.e.\
 $u=\frac{1-n_1/N}{1-w}+x_1/N+n_1/N-1$ transforms the leading
 asymptotic of \eqref{eq_Kernel_bulk_asymptotic} into
\begin{equation}
\label{eq_Incomplete_beta} \frac{1}{2\pi \ii} \int_{ \bar \xi }^{\xi} w^{x_2-x_1-1}
(1-w)^{n_2-n_1}
  dw,
\end{equation}
where the integration contour intersects the real line inside $(0,1)$ when $n_2\ge
n_1$ and inside $(-\infty,0)$ otherwise;
$$
 \xi= \frac{\tau_1/N-x_1/N}{\tau_1/N-x_1/N+1-n_1/N}.
$$
The form \eqref{eq_Incomplete_beta} is known as the incomplete Beta-kernel.
\end{remark}

In the rest of this section we prove Theorem \ref{Theorem_Kernel_asympt}.

\smallskip

The level lines of functions $\Re G_\kappa(z)$ and $\Im G_\kappa(\tau)$ passing
through $\tau_\kappa$, $\kappa=1,2$, are important for what follows, they are
schematically sketched in Figure \ref{Fig_real_part}.

\begin{figure}[h]
\begin{center}
 {\scalebox{0.9}{\includegraphics{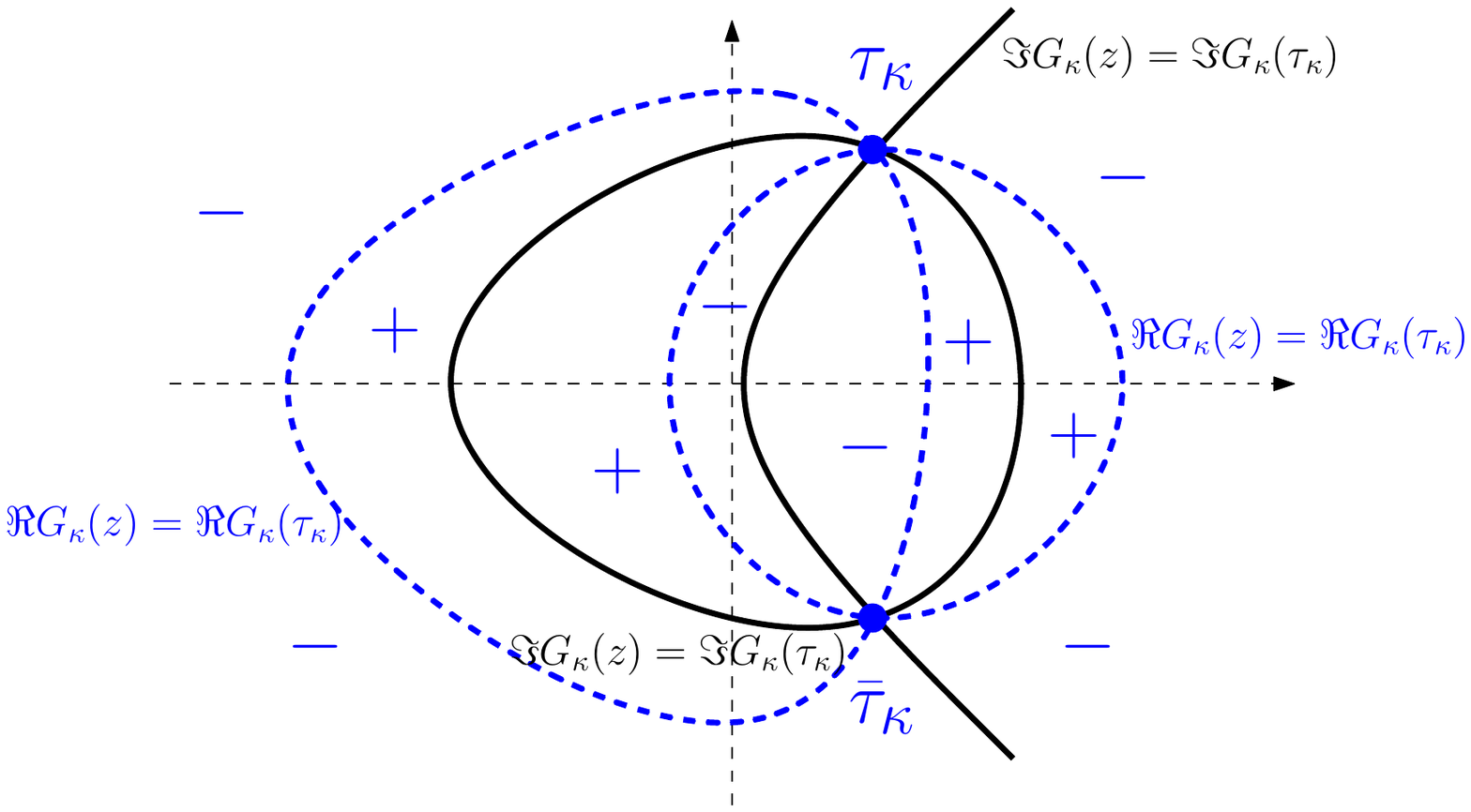}}}
\end{center}
 \caption{Contours $\Re G_\kappa(z)
 =\Re G_\kappa(\tau)$, $\Im G_\kappa(z)=\Im G_\kappa(\tau_\kappa)$ going through $\tau_\kappa$, $\bar \tau_\kappa$. Pluses and minuses indicate
 the sign of $\Re(G_\kappa(z)-G_\kappa(\tau))$ in each domain. \label{Fig_real_part}}
\end{figure}

Let us explain the key features of Figure \ref{Fig_real_part}. Since $\tau_\kappa$ and $\bar
\tau_\kappa$ are simple critical points of $G_\kappa(z)$ (i.e.\ $G''_\kappa(\tau_\kappa)\ne 0$,
$G''_\kappa(\tau_\kappa)\ne 0$) there are four branches of $\Re G_\kappa(z)=const$ and four
branches of $\Im G_\kappa(z)=const$ going out of each critical point and these two kinds of
branches interlace.

Since $\Re(G_\kappa(z)) \to -\infty$ as $z\to\infty$ the level lines of $\Re
G_\kappa(z)$ are closed curves. These level lines can not intersect in a non-real
point except at $\tau_\kappa$, $\bar \tau_\kappa$, as such an intersection would
have been another critical point for $G_\kappa(z)$. Further, since $\Re G_\kappa(z)$
is harmonic everywhere outside the singularities on the real axis, each of its
closed level lines must have one of these singularities inside. Finally,
$\Re(G_\kappa(z))=\Re(G_\kappa(\bar z))$ and therefore, the level lines are
symmetric. Combination of these properties implies that there are four
non-intersecting curves of constant $\Re G_\kappa(z)$ joining $\tau_\kappa$ with
$\bar \tau_\kappa$, as in Figure \ref{Fig_real_part}.

Let us now in addition draw all other level lines $\Re G_\kappa(z)=\Re G_\kappa(\tau_\kappa)$;
there might be no others in addition to the 4 we just drawn, but also there might be some
additional loops near the real axis.

We would like to distinguish four regions on the plane bounded by level lines of
$\Re G_\kappa(z)$; these are 4 regions adjacent to $\tau_\kappa$. (Note that, in
principle, these four regions do not have to cover the whole plane, but we know that
each of them contains some points of the real axis). One of these regions is
unbounded, while the other two are bounded and we call them left, middle and right
(according to the position of their intersections with the real axis) Due to the
maximum principle for the harmonic functions and the behavior at the infinity of
$\Re(G_\kappa(z))$ in the exterior (infinite) region we have $\Re G_\kappa(z)< \Re
G_\kappa(\tau_\kappa)$. Therefore, since $\tau_\kappa$ is a simple critical point
for $G_\kappa(z)$, also $\Re G_\kappa(z)< \Re G_\kappa(\tau_\kappa)$ in the middle
bounded region, and $\Re G_\kappa(z)> \Re G_\kappa(\tau_\kappa)$ in both left and
right bounded regions.

We will further distinguish four contours. $C^{(m)}_\kappa$ is the curve
$\Im(G_\kappa(z))=\Im(\tau_\kappa)$ which starts at $\tau_\kappa$ and continues in
the middle bounded region until it reaches the real axis. After that we continue the
curve symmetrically to $\bar \tau_\kappa$. Note that $\Re (G_\kappa(z))$ is
monotonous along $C^{(m)}_\kappa$ except, perhaps, at the real axis, since its local
extremum would necessary be a critical point. Therefore, $\Re (G_\kappa(z))$
monotonously decreases as we move away from $\tau_\kappa$ in the upper half--plane
and increases as we move towards $\bar \tau_\kappa$ in the lower halfplane. Also
observe that the choice of the branch of the logarithm in
\eqref{eq_action_logarithm} is not important here as the curve $C^{(m)}_\kappa$ is
constructed locally and therefore does not depend on this choice.

Further, $C^{(r)}_\kappa$ is the curve $\Im(G_\kappa(z))=\Im(\tau_\kappa)$ which starts at
$\tau_\kappa$ and continues in the right bounded region until it reaches the real axis. After that
we continue the curve symmetrically to $\bar \tau_\kappa$. Along this curve $\Re (G_\kappa(z))$
monotonously increases as we move away from $\tau_\kappa$ in the upper half--plane and decreases as
we move towards $\bar \tau_\kappa$ in the lower halfplane.

The contour $C^{(l)}_\kappa$ is defined in the same way, but inside the left bounded region.

The contour $C^{(i)}_\kappa$ is the curve $\Im(G_\kappa(z))=\Im(\tau_\kappa)$ which starts at
$\tau_\kappa$ and continues in the infinite region. This curve will never go back to the real axis
(this is established below by counting the points on the real axis where
$\Im(G_\kappa(z))=\Im(\tau_\kappa)$) and thus goes to infinity in the upper half--plane. We then
return symmetrically from the infinity to $\bar\tau_\kappa$ in the lower half--plane.

\begin{lemma} \label{Lemma_contour_positioning}
 The curve $C^{(m)}_\kappa$ intersects the real axis inside the interval
 $(x_\kappa+n_\kappa-N,x_\kappa-1)$, and the curve $C^{(i)}_\kappa$ does not intersect the real axis.
 The curve $C^{(l)}_\kappa$ intersects the real axis to the left from $x_\kappa+n_\kappa-N$, and
 the curve $C^{(r)}_\kappa$ intersects the real axis to the right from $x_\kappa-1$.
\end{lemma}
\begin{proof}
 Choose a small positive parameter $\eps>0$ and let $\gamma$ be the curve which is
 the real axis, except near the singularities of $G(z)$ (i.e.\ the points $d$
 and $d'$ in \eqref{eq_critical_rewritten}). $\gamma$ walks around each singularity
 by a half--circle in the upper-halfplane centered at this singularity and of the
 radius $\eps$.

 We remarked above that the curves $C^{(m)}_\kappa$, $C^{(r)}_\kappa$, $C^{(l)}_\kappa$, and
 $C^{(i)}_\kappa$ do not depend on the choice of branches of the logarithm in
 \eqref{eq_action_logarithm}. But we need to choose \emph{some} branch and so we use
 the branch of $\ln(y)$ with a cut along the negative \emph{imaginary} axis in terms of $y$
 and such that the value at positive real $y$ is real. With this notation in mind,
 let us trace the values of $\Im(G_\kappa(z))$ along the curve $\gamma$ moving in the direction of growing $x$. Clearly, if we choose $\eps$ small enough, then $\Im(G_\kappa(z))$
 is constant on straight segments between the singularities, increases along the
 half--circles corresponding to singularities $d$ in \eqref{eq_critical_rewritten}
 and decreases along the half--circles corresponding to singularities $d'$ in
 \eqref{eq_critical_rewritten}. At $+\infty$ the value of this function is $0$. This is schematically illustrated in Figure \ref{Fig_imaginary_part}.

\begin{figure}[h]
\begin{center}
 {\scalebox{0.9}{\includegraphics{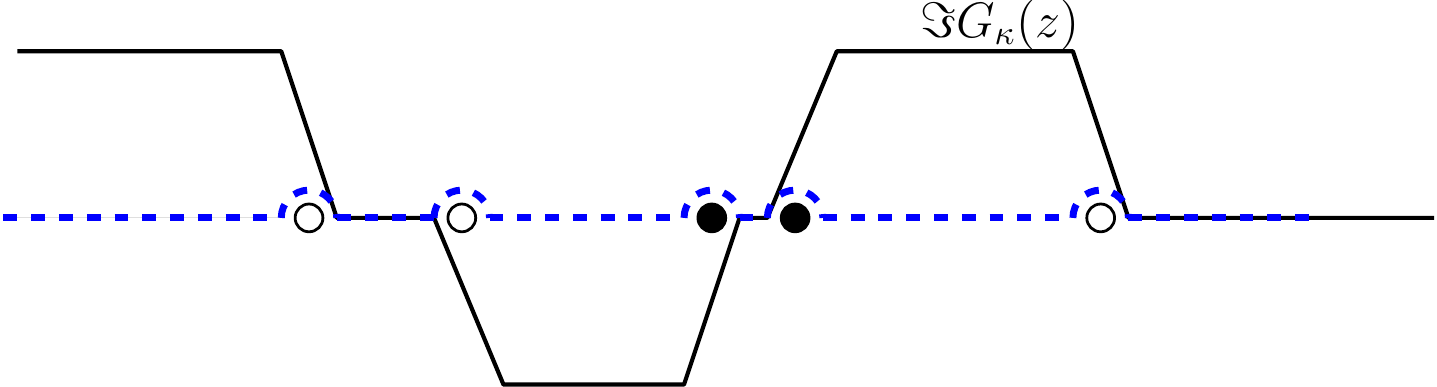}}}
\end{center}
 \caption{Imaginary part (in solid black) of $G_\kappa(z)$ along the (blue dashed) curve which follows the real line avoiding the
 singularities. White points correspond to $d$ terms in the first sum of
 \eqref{eq_critical_rewritten} and black points correspond to $d'$ terms in the
 second sum.
 \label{Fig_imaginary_part}}
\end{figure}

At a point $z$ where one of our curves intersects the real axis we should have $\Im
(G_\kappa(z))=\Im (G_\kappa(\tau_\kappa))$. We claim that there can not be two intersections on the
same horizontal segment of the graph of $\Im (G_\kappa(z))$ along $\tau$. Indeed, each such
intersection is a critical point of $G_\kappa(z)$ (since the real axis itself is also a level line
$\Im (G_\kappa(z))=const$), but the analysis of critical points in the proof of Lemma
\ref{Lemma_number_of_roots} shows that there are never more than one critical points between two
real singularities of $G_\kappa$. Also these curves can not intersect the real axis on the segments
$(d^l_L, d'_1)$ and $(d'_P, d_1^R)$
--- because we know that there are no critical points of $G_1$ on these segments.

Therefore, our curves can not have more than $3$ intersections with real axis, and, thus, they have
exactly three, in particular, $G^{(i)}_\kappa$ does not intersect the real axis. The middle
intersection then belongs to $G^{(m)}_\kappa$ and should be in the interval $[d'_1, d'_P]\subset
(x_\kappa+n_\kappa-N,x_\kappa-1)$. The right intersection belongs to $G^{(r)}_\kappa$ and should be
in the interval $[d^r_1,+\infty)\subset (x_\kappa-1,+\infty)$. Finally, the left intersection
belongs to $G^{(l)}_\kappa$ and should be in the interval $(-\infty, d^l_L]\subset (-\infty,
x_\kappa+n_\kappa-N)$.
\end{proof}

Next, we need analogues of Lemmas \ref{Lemma_number_of_roots},
\ref{Lemma_contour_positioning}, in which the point configuration $\t$ is replaced
by a probability measure. Take a probability density $\nu(x)$, $x\in\mathbb R$ such
that $\nu(x)\le 1$ for all $x$, and two numbers, $\mathbf x\in\mathbb R$,
$0<\boldsymbol{\eta}<1$. Given this data, define
\begin{equation}
\label{eq_action_continuous}
 G(z)=\int_{\mathbf x+\boldsymbol{\eta} -1}^{\mathbf x} \ln(z-t) dt -\int_{\mathbb R} \ln(z-y) \nu(y) dy.
\end{equation}
\begin{lemma} \label{Lemma_number_of_roots_cont}
 The equation $G'(z)=0$ has either $0$ or $2$ non-real roots. In the latter case the non-real roots
 are complex conjugate to each other.
\end{lemma}
\begin{proof}
 The function $G(z)$ outside the real axis can be (uniformly on the compact subsets of $\mathbb C\setminus \mathbb R$) approximated by functions of the form $\frac{1}{N}
 G_\kappa(z)$ of \eqref{eq_action_logarithm}. It remains to combine Lemma
 \ref{Lemma_number_of_roots} with Rouche's theorem.
\end{proof}
Suppose that $G'(z)=0$ has two complex roots, and let $\tau$ denote the root in the upper
half-plane. Define four curves $C^{(m)}$, $C^{(l)}$, $C^{(r)}$, $C^{(i)}$ in the same way as
$C^{(m)}_\kappa$, $C^{(l)}_\kappa$, $C^{(r)}_\kappa$, $C^{(i)}_\kappa$, i.e.\ they are parts of the
lines $\Im(G_\kappa(z))=\Im(\tau_\kappa)$ in the upper half--plane continued symmetrically to the
lower half--plane.
\begin{lemma} \label{Lemma_contour_positioning_cont}
 The curve $C^{(m)}$ intersects the real axis inside the interval
 $(\mathbf x+ \boldsymbol{\eta}_\kappa-1,\mathbf x)$, and the curve $C^{(i)}$ does not intersect the real axis.
 The curve $C^{(l)}$ intersects the real axis to the left from $\mathbf x+ \boldsymbol{\eta}-1$, and
 the curve $C^{(r)}$ intersects the real axis to the right from $\mathbf x$.
\end{lemma}
\begin{proof}
 The proof follows the same lines as that of Lemma \ref{Lemma_contour_positioning}. In more details,
  we trace
 the imaginary part of $G(z)$ along the curve $\Im(z)=\eps>0$.
\end{proof}
The last technical ingredient needed for the proof of Theorem
\ref{Theorem_Kernel_asympt} is the following statement.

\begin{lemma} \label{Lemma_bound_derivative}
 For $N=1,2,\dots$, set
 $$R_N(x)=\sum_{k=-N}^N \left|\frac{1}{\ii x  +1/2 +k}\right|=\sum_{k=-N}^N \frac{1}{\sqrt{x^2
 +(k+1/2)^2}}.
 $$
 Then
 $$
 \lim_{\eps\to 0} \limsup_{N\to\infty} \frac{1}{N}\int_0^{N\eps} R_N(x) dx=0.
 $$
\end{lemma}
\begin{proof}
 Integrating $R_N(x)$ termwise we get
 \begin{equation}
 \label{eq_integrated_derivative_sum}
  \int_0^{N\eps} R_N(x)dx =\sum_{k=-N}^N \arcsh\left(\frac{\eps N}{|k+1/2|}\right),
 \end{equation}
 where $\arcsh(x)$ is the inverse functions to $\sinh(x)=\frac{e^x-e^{-x}}{2}$, i.e.\
 $$
  \arcsh(x)=\ln\left(x+\sqrt{x^2+1}\right),\quad x\ge 0.
 $$
 Split the sum in \eqref{eq_integrated_derivative_sum} into two: the first one has $k$ such that
 $|k+1/2|>\sqrt{\eps} N$ and the second one is the rest.
The first sum is bounded from above by
\begin{equation}
\label{eq_first_sum_bound}
 (2N+1)\arcsh(\sqrt{\eps}).
\end{equation}
For the second sum we use the inequality
$$
 \arcsh(x)= \ln(x)+\ln\left(1+\sqrt{1+1/x^2}\right)\le
 \ln(x)+\ln\left(\frac{3}{\sqrt{\eps}}\right),\quad x>\sqrt{\eps},\quad 0<\eps<1,
$$
which gives the bound
\begin{multline}\label{eq_second_sum_bound}
 \sum_{|k+1/2|\le \sqrt{\eps} N} \arcsh\left(\frac{\eps N}{|k+1/2|}\right)\le \sum_{|k+1/2|<\sqrt{\eps} N}\left( \ln(\eps)+\ln(N)-\ln|k+1/2| +
 \ln\left(\frac{3}{\sqrt{\eps}}\right)\right)\\ \le (2\sqrt{\eps} N+1)( \ln(N) + \ln(3/\sqrt{\eps}) )
 -2\ln\bigl( \lfloor \sqrt{\eps} N -1 \rfloor ! \bigr)
\end{multline}
The Stirling's formula yields $\ln(x!)>x(\ln(x)-1)$ for large $x$, and therefore, we further bound
\eqref{eq_second_sum_bound} by
\begin{multline} \label{eq_second_sum_final_bound}
(2\sqrt{\eps} N+1)( \ln(N) + \ln(3/\sqrt{\eps}) )
 -2 (\sqrt{\eps} N -2)\bigl(\ln(\sqrt{\eps} N-2)-1\bigr)\\ \le C( \ln(N) + \sqrt{\eps} N +
 \sqrt{\eps}\ln(1/\eps) N + 1),
\end{multline}
for a constant $C>0$. Summing \eqref{eq_first_sum_bound} and \eqref{eq_second_sum_bound}, dividing
by $N$, sending $N\to\infty$ and then $\eps\to 0$ we get the desired claim.
\end{proof}

\smallskip

\begin{proof}[Proof of Theorem \ref{Theorem_Kernel_asympt}] We denote through $\mu_N$ the signed measure on $\mathbb R$ describing the
scaled by $N$ points $d$ and $d'$ from \eqref{eq_critical_rewritten}:
\begin{equation}
\label{eq_measures}
 \mu_N= \frac{1}{N} \sum_{d\in \T \setminus [x_1-N+ n_1, x_1-1] } \delta_{d/N} - \frac{1}{N} \sum_{d'\in \Tc \bigcap [x_1-N+ n_1, x_1-1]}
 \delta_{d'/N}.
\end{equation}
The assumption $\sum_{i=1}^N \ln^{1+1/D}(1+|\t_i|/N)<DN$ implies that the measures $\mu_N$ are
tight, and therefore by the standard compactness arguments we can (and will) assume that the
measures $\mu_N$ weakly converge as $N\to\infty$ to a signed measure $\mu$. Note that $\mu_N$--mass
of any interval $[a,b]$ is satisfies the bounds
$$
a-b-\frac{1}{N}\le \mu_N([a,b])\le b-a+\frac{1}{N},
$$
therefore $\mu$ is an absolutely continuous signed measure with density $\mu(x)$ bounded between
$-1$ and $1$.

Recall the functions $G_i(z)$ of \eqref{eq_action_logarithm},  and let us introduce
their normalized versions through:
\begin{multline*}
 G_1^{(N)}(u)=\frac{1}{N} \sum_{d\in \T \setminus [x_1-N+ n_1, x_1-1] } \ln(u-d/N) -
 \frac{1}{N} \sum_{d'\in \Tc \bigcap [x_1-N+ n_1, x_1-1]}
 \ln(u-d'/N)\\=  \int_{\mathbb R} \ln(z-x) \mu_N(dx),
\end{multline*}
and similarly for $G_2^{(N)}(u)$. Also set
$$
 G^{\infty}(z)= \int_{\mathbb R} \ln(z-x) \mu(x) dx,
$$
which is well-defined due to the assumption $\sum_{i=1}^N \ln^{1+1/D}(1+|\t_i|/N)<DN$, and note
that the same assumption implies
$$
 G_1^{(N)}(z)\rightrightarrows G^{\infty}(z)
$$
uniformly over $z$ in compact subsets of $\mathbb C\setminus\{z\mid \Im(z)=0\}$. Since the
differences $|n_2-n_1|$ and $|x_1-x_2|$ stay finite as $N\to\infty$, we also have
$$
 G_2^{(N)}(z)\rightrightarrows G^{\infty}(z).
$$

We now introduce two integration contours. For the $\C_z$ contour we start from the
union of $C^{(m)}_1$ and $C^{(i)}_1$ oriented from top to bottom. From the technical
point of view it is convenient to modify this contour near the real
axis\footnote{The subsequent proofs will probably go through even without this
modification, but rigorous justifications of some steps would become more involved,
due to the singularities of the integrand near the real axis.} Namely, we choose
$\eps>0$ to be fixed later and assume that $\Re \tau_1
>\eps N$. As soon as $\C_z$ contour (recall that be are tracing it from the bottom) reached the
level $\Im(z)=-\eps N$, it immediately has a horizontal segment so that to turn
$\Re(z)$ into a half--integer (i.e.\ number of the form $1/2+n$, $n\in\mathbb Z$),
and then vertical segment from $\Im(z)=-\eps N$ to $\Im(z)=0$. In the upper
halfplane $\Im(z)>0$ we do the same modification. Similarly, and $\C_w$ is defined
as the union of $C^{(r)}_2$ and $C^{(l)}_2$ oriented counter--clockwise and modified
to a vertical line with half--integer real part for $-\eps N< \Im(z)<\eps N$. We
refer to Figure \ref{Fig_contour_modification} for an illustration.

\begin{figure}[h]
\begin{center}
 {\scalebox{0.7}{\includegraphics{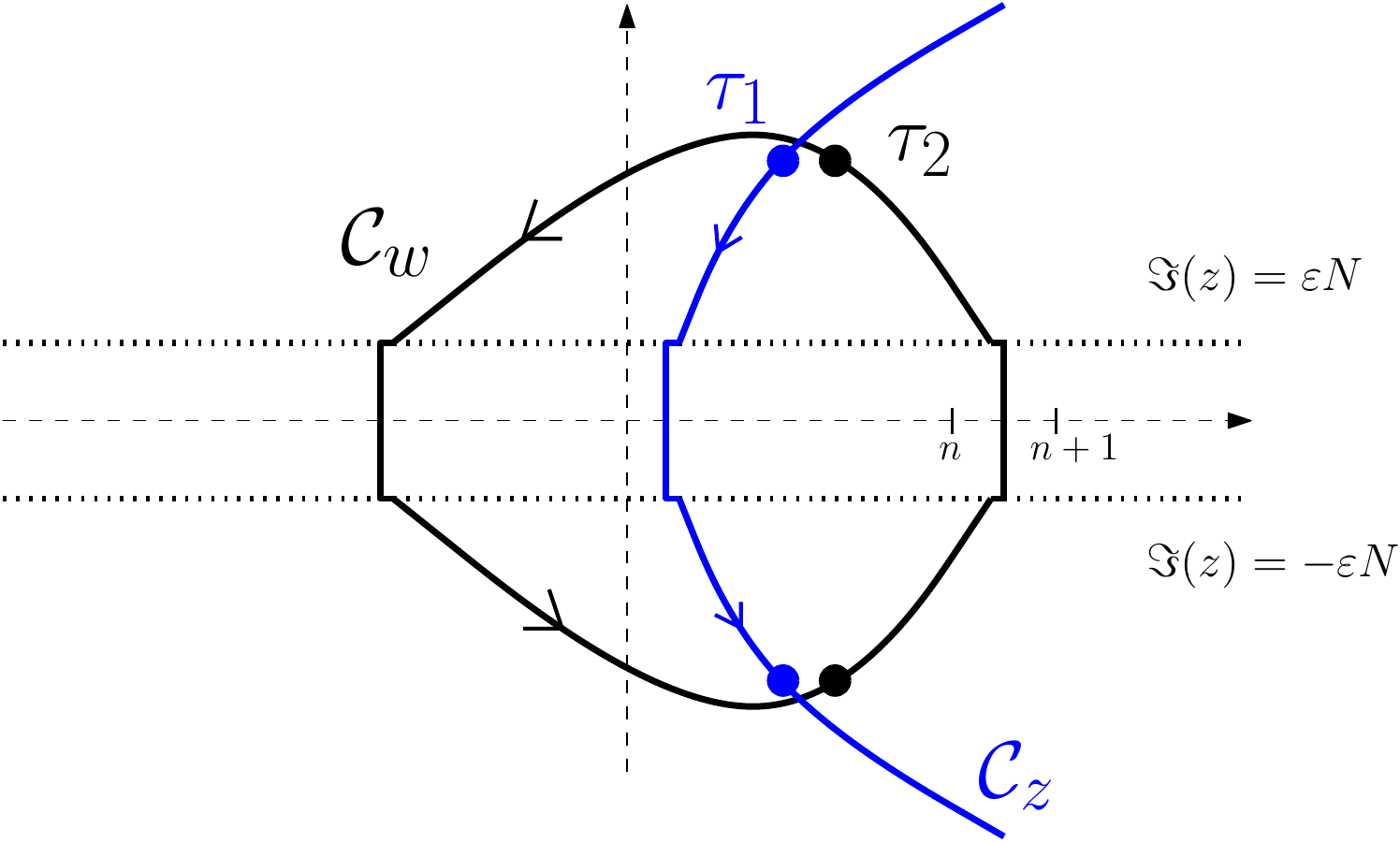}}}
\end{center}
 \caption{$\C_z$ and $\C_w$ contours, which are the level lines of $\Im(G_{1,2}(z))$ (cf.\ Figure \ref{Fig_real_part}) modified near the real axis. \label{Fig_contour_modification}}
\end{figure}

Note that due to convergence $G_{1,2}^{(N)}(z)\rightrightarrows G^{\infty}(z)$, the contours do not
oscillate as $N\to\infty$, but smoothly approximate similar contours constructed using
$G^{\infty}(z)$.

We next deform the $z$ and $w$ contours in \eqref{eq_correlation_kernel} into $\C_z$
and $\C_w$, respectively. Lemmas \ref{Lemma_contour_positioning},
\ref{Lemma_contour_positioning_cont} imply that the only residues which we collect
in this deformation are those coming from the $\frac{1}{w-z}$ pole. Let us first
deform the $z$ contour (there is no residue coming from $w=z$ at this stage, as $w$
is very large), and then proceed to the $w$--contour. Therefore, the result of the
deformation is the $z$--integral over a contour $\C^{def}_z\subset \C_z$, of the
$w$--residue of the integrand in the double integral in
\eqref{eq_correlation_kernel} at point $w=z$, i.e.\
\begin{multline}
\label{eq_Cor_kernel_contours_transformed}
 K(x_1,n_1;x_2,n_2)= -\1_{n_2<n_1} \1_{x_2\le x_1}
 \frac{(x_1-x_2+1)_{n_1-n_2-1}}{(n_1-n_2-1)!}\\+ \frac{(N-n_1)!}{(N-n_2-1)!}\frac{1}{2\pi\ii} \int_{\C^{def}_z} \frac{dz}{(z-x_1)(z-x_2+N-n_2)} \cdot
\frac{\exp(G_2(z)}{\exp (G_1(z))}\\ +
 \frac{(N-n_1)!}{(N-n_2-1)!} \frac{1}{(2\pi \ii)^2}
\oint_{\C_z} dz \oint_{\C_w} dw \frac{1}{w-z} \cdot \frac{1}{(w-x_1)(z-x_2+N-n_2)}
\cdot \frac{\exp(G_2(z))}{\exp(G_1(w))}.
\end{multline}
Note that $\C^{def}_z$ is passed from bottom to top and it intersects the real line
to the right from $x_1$.

\smallskip

The change of variables $u=z/N$, $v=w/N$ transforms the kernel $K(x_1,n_1;x_2,n_2)$
into the form
\begin{multline} \label{eq_Cor_kernel_variables_changed}
 K(x_1,n_1;x_2,n_2)= -\1_{n_2<n_1} \1_{x_2\le x_1}
 \frac{(x_1-x_2+1)_{n_1-n_2-1}}{(n_1-n_2-1)!}\\+ \frac{(N-n_2)_{n_2+1} N^{n_1}}{(N-n_1+1)_{n_1} N^{n_2+1}}\frac{1}{2\pi\ii} \int_{\C^{def}_u} \frac{du}{(u-x_1/N)(u-x_2/N+1-n_2/N)} \cdot
\frac{\exp(N G_2^{(N)}(u))}{\exp(NG_1^{(N)}(u))}\\ +
 \frac{(N-n_2)_{n_2+1} N^{n_1}}{(N-n_1+1)_{n_1} N^{n_2+1}} \frac{1}{(2\pi \ii)^2}
\oint_{\C_u} du \oint_{\C_v}  \frac{dv}{v-u} \cdot
\frac{1}{\left(v-\frac{x_1}{N}\right)\left(u-\frac{x_2}{N}+1-\frac{n_2}{N}\right)}
\frac{\exp(NG_2^{(N)}(u))}{\exp(NG_1^{(N)}(v))},
\end{multline}
where $u$ and $v$ contours are the contours of
\eqref{eq_Cor_kernel_contours_transformed} rescaled by $N$. Note that since the
functions $G_1^{(N)}$ and $G_2^{(N)}$ uniformly converge to $G^{\infty}$ outside
$\eps$--neighborhood of the real line, and the contours are the level lines of the
imaginary part of the former functions, they converge to similar level lines for
$G^{\infty}$. In particular, the lengths of all the involved contours are bounded as
$N\to\infty$, and $\C_u$ might intersect with $\C_v$ only in a neighborhood of
$\tau_1/N$.

Let us analyze the $N\to\infty$ behavior of each term in
\eqref{eq_Cor_kernel_variables_changed}. The first line does not depend on $N$. For
the second line observe that for $x_1\ge x_2$, $n_1\ge n_2$ we have
\begin{multline}
\label{eq_ratio_G_functions}
 \frac{\exp(NG_2^{(N)}(u))}{\exp(NG_1^{(N)}(u))}=\frac{\prod\limits_{a=1}^{N-n_1 + (n_1-n_2)}\left(u-\frac{x_1}{N}+\frac{a}{N}+\frac{x_1-x_2}N\right)}
 {\prod\limits_{a=1}^{N-n_1}\left(u-\frac{x_1}{N}+\frac{a}{N}\right)}
\\=\frac{\prod\limits_{a=N-n_1+1}^{N-n_1+(n_1-n_2)+(x_1-x_2)}\left(u-\frac{x_1}{N}+\frac{a}{N}\right)}
{\prod\limits_{a=1}^{x_1-x_2}\left(u-\frac{x_1}{N}+\frac{a}{N}\right)}.
\end{multline}
Due to Lemmas \ref{Lemma_contour_positioning}, \ref{Lemma_contour_positioning_cont}
 the integration contours are
bounded away from the points $u=x_1/N$ and $u=(x_1+n_1-N)/N$. Therefore,
\eqref{eq_ratio_G_functions} uniformly converges on the integration contours and
thus the second line in \eqref{eq_Cor_kernel_variables_changed} is as $N\to\infty$
$$
 \frac{(1-n_1/N)^{1+n_2-n_1}}{2\pi\ii} \int_{\C^{def}_u}
 \frac{\left(u-\frac{x_1}{N}\right)^{x_2-x_1-1}}{
 \left(u-\frac{x_1}{N}+1-\frac{n_1}{N}\right)^{1+n_2-n_1+x_2-x_1}} du + o(1).
$$
When $x_1<x_2$ or $n_1<n_2$ the computation is the same. Since the point of the
intersection of the contours $\C_u$ and $\C_v$ approaches $\tau_1/N$ as
$N\to\infty$, the final asymptotic for the second line of
\eqref{eq_Cor_kernel_variables_changed} is
\begin{equation}
 \frac{(1-n_1/N)^{1+n_2-n_1}}{2\pi\ii} \int_{\bar \tau_1/N}^{\tau_1/N}
 \frac{\left(u-\frac{x_1}{N}\right)^{x_2-x_1-1}}{
 \left(u-\frac{x_1}{N}+1-\frac{n_1}{N}\right)^{1+n_2-n_1+x_2-x_1}} du + o(1),
\end{equation}
where the integration contour crosses the real axis to the right from
$\frac{x_1}{N}$.

Now we turn to the third line in \eqref{eq_Cor_kernel_variables_changed}. Fix any
$\delta>0$. We claim that outside the $\delta$--neighborhood of the points
$\tau_1/N$, $\bar \tau_1/N$ the integrand is exponentially (in $N$) small. Indeed,
by the construction of the contours, $\Re( G_2^{(N)}(u)$ (strictly) decreases as we
move away from the point $\tau_2/N$ (similarly with $\bar \tau_2/N$) as long as we
do not get into the $\eps$--neighborhood of the real axis. Lemma
\ref{Lemma_bound_derivative} gives a uniform bound for the derivative of
$G_2^{(N)}(u)$ near the real axis, which shows, that $\Re( G_2^{(N)}(u)$ can not
grow much near the real axis.

In the same way $\Re( G_1^{(N)}(u))$ increases as we move away from $\tau_1/N$ and
does not grow much near the real axis, where monotonicity no longer holds.

On the other hand for small values of $\delta$, we can Taylor expand $G_1^{(N)}(u)$
and $G_2^{(N)}(u)$ near the points $\tau_1/N$, $\tau_2/N$, respectively. Since
$G_i^{(N)}(u)$ uniformly converges to $G^{\infty}(u)$, we essentially deal with the
Taylor expansion of the latter function. Since $\tau_1/N$, $\tau_2/N$ are critical
points of the corresponding functions, we have
\begin{multline} \label{eq_near_critical_expansion}
 \exp\bigl(N(G_2^{(N)}(u)-G_1^{(N)}(v))\bigr)=\exp\bigl(N(G_2^{(N)}(\tau_2/N)-G_1^{(N)}(\tau_1/N))\bigr)
\\
\times \exp\Bigl(N \bigl(  (G_2^{(N)})'' (\tau_2/N) (u-\tau_2/N)^2 - (G_1^{(N)})''
(\tau_1/N) (v-\tau_1/N)^2 + O( |u-\tau_2/N|^3+ |v-\tau_1/N|^3) \bigr) \Bigr)
\end{multline}
Since the difference $G_1^{(N)}(u)-G_2^{(N)}(u)$ is (uniformly) of order $1/N$ as
$N\to\infty$, the factor
$\exp\bigl(N(G_2^{(N)}(\tau_2/N)-G_1^{(N)}(\tau_1/N))\bigr)$ stays bounded as
$N\to\infty$. Turning to the second line in \eqref{eq_near_critical_expansion}, note
that by the definition, the tangent to the $\C_u$ contour at $\tau_1/N$ is such that
on this tangent $(G_2^{(N)})'' (\tau_2/N) (u-\tau_2/N)^2$ is negative real and the
tangent to the $\C_v$ contour at $\tau_2/N$ is such that $(G_1^{(N)})'' (\tau_1/N)
(v-\tau_1/N)^2$ is positive real. Therefore, the integral (over
$\delta$--neighborhood of $\tau_1/N$, along our contours) in the third line of
\eqref{eq_Cor_kernel_variables_changed} decays as $N\to\infty$ (in fact, it behaves
as $O(N^{-1/2})$).

\bigskip

Summing up, \eqref{eq_Cor_kernel_variables_changed} behaves when $N\to\infty$ as
\begin{multline} \label{eq_asymptotic_bulk_two_terms}
 K(x_1,n_1;x_2,n_2)= -\1_{n_2<n_1} \1_{x_2\le x_1}
 \frac{(x_1-x_2+1)_{n_1-n_2-1}}{(n_1-n_2-1)!}\\+
 \frac{(1-n_1/N)^{1+n_2-n_1}}{2\pi\ii} \int_{\bar \tau_1/N}^{\tau_1/N}
 \frac{\left(u-\frac{x_1}{N}\right)^{x_2-x_1-1}}{
 \left(u-\frac{x_1}{N}+1-\frac{n_1}{N}\right)^{1+n_2-n_1+x_2-x_1}} du + o(1),
\end{multline}
where the integration contour crosses the real axis to the right from
$\frac{x_1}{N}$.

We claim that when $n_2<n_1$, $x_2\le x_1$, then the first term in
\eqref{eq_asymptotic_bulk_two_terms} is precisely the minus residue of the second
term at $u=x_1/N$. Indeed, from one side
$$
  \frac{(x_1-x_2+1)_{n_1-n_2-1}}{(n_1-n_2-1)!}=
  \frac{(x_1-x_2+n_1-n_2-1)!}{(x_1-x_2)!(n_1-n_2-1)!}.
$$
On the other side,
\begin{multline}
{\rm Res}_{u=\frac{x_1}{N}}  \left[
 \frac{\left(u-\frac{x_1}{N}+1-\frac{n_1}{N}\right)^{n_1-n_2+x_1-x_2-1}}
 {\left(u-\frac{x_1}{N}\right)^{x_1-x_2+1}} \right]\\ = \frac{1}{(x_1-x_2)!}\left(\frac{\partial}{\partial
 u}\right)^{x_1-x_2}
 \left(u-\frac{x_1}{N}+1-\frac{n_1}{N}\right)^{n_1-n_2+x_1-x_2-1} \Biggr|_{u=x_1/N}
\\ = \frac{(n_1-n_2+x_1-x_2-1)!}{(x_1-x_2)!(n_1-n_2-1)!}
\left(1-\frac{n_1}{N}\right)^{n_1-n_2-1}.
\end{multline}
We conclude that
\begin{multline} %\label{eq_asymptotic_bulk_two_terms}
 K(x_1,n_1;x_2,n_2)= \frac{(1-n_1/N)^{1+n_2-n_1}}{2\pi\ii} \int_{\bar \tau_1/N}^{\tau_1/N}
 \frac{\left(u-\frac{x_1}{N}\right)^{x_2-x_1-1}}{
 \left(u-\frac{x_1}{N}+1-\frac{n_1}{N}\right)^{1+n_2-n_1+x_2-x_1}} du + o(1),
\end{multline}
where the integration contour crosses the real axis to the right from $x_1/N$ when
$n_2\ge n_1$ and inside the interval $(\frac{x_1+n_1-1}{N},\frac{x_1}{N})$
otherwise.
\end{proof}

\section{Law of Large Numbers for tilings}
\label{Section_LLN}

Take a lozenge tiling of an arbitrary \emph{simply-connected} domain on a regular
triangular grid. We aim to define for each vertex $\vv$ of the grid inside the
domain the value of  \emph{height function} $H(\vv)$. For that we choose two types
of lozenges out of three (we have chosen \emph{non-horizontal} ones, see Figure
\ref{Fig_paths_height}) and draw their middle lines, thus arriving at a family of
non-intersecting paths corresponding to the tiling. The value of the height function
increases by $1$ when we cross such a path (from bottom to top), this condition
defines the values of $H(\cdot)$ up to an addition of an arbitrary constant. We fix
this constant by a convention that $H$ vanishes at the bottom point of the left-most
vertical of the domain, see Figure \ref{Fig_paths_height}.

\begin{figure}[h]
\begin{center}
 {\scalebox{0.9}{\includegraphics{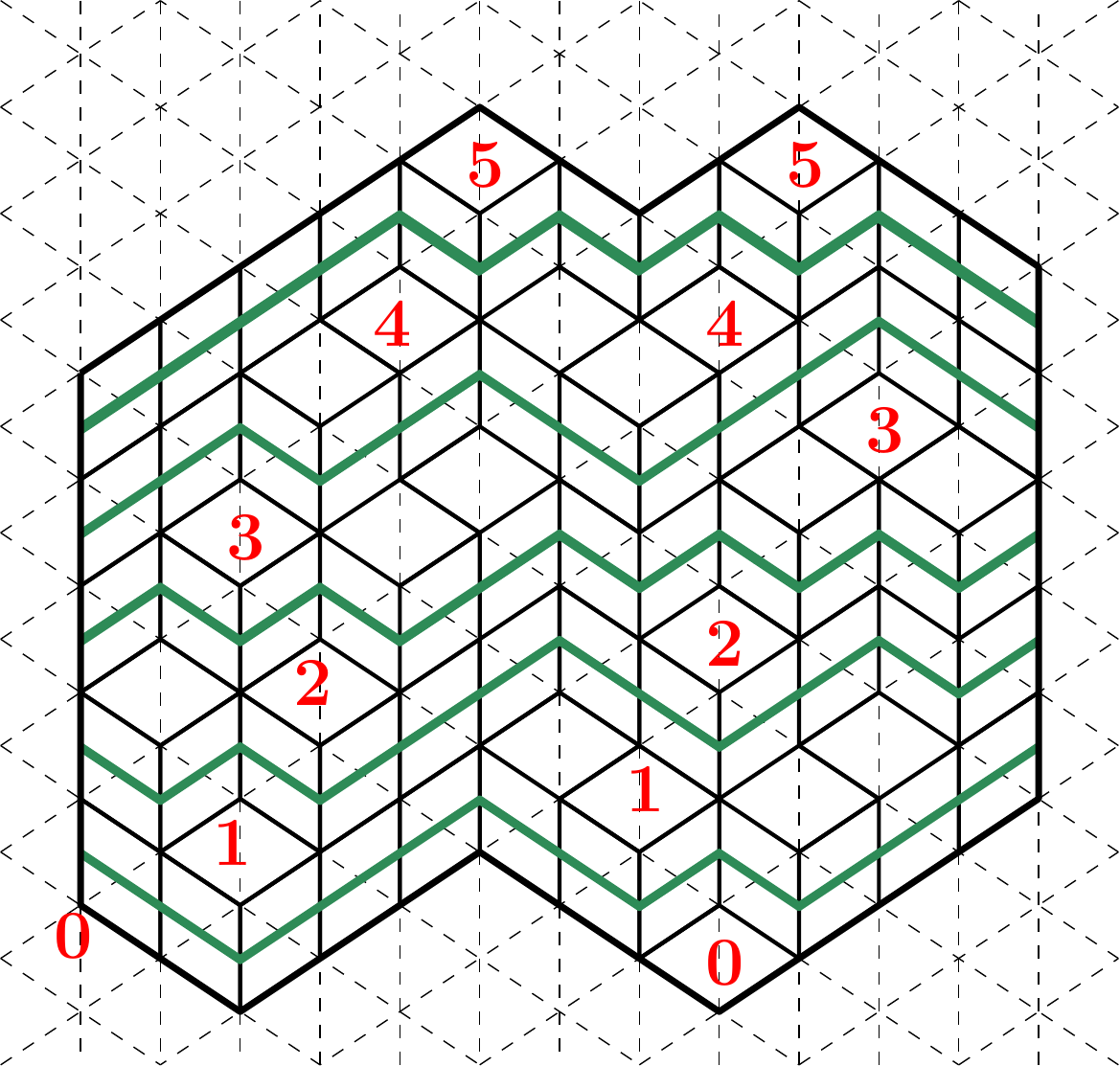}}}
\end{center}
 \caption{Height function of a lozenge tiling.\label{Fig_paths_height}}
\end{figure}

Observe that along the boundary of the domain, the values of the height function do
not depend on the choice of lozenge tiling. Indeed, the height function changes
linearly along the vertical segments of the boundary and is constant along two other
types of boundary segments.

Let $\Omega_L$, $L=1,2,\dots$, be a sequence of simply--connected domains on the
triangular grid such that each $\Omega_L$ has at least one lozenge tiling. Further,
let $\Omega$ be a simply--connected domain with piecewise-smooth boundary $\partial
\Omega$, and let $h$ be a continuous real function on $\partial \Omega$.

For two sets $A,B\subset \mathbb R^k$ we say that they are within $\eps$--distance
from each other, if for each $x\in A$ there exists $y\in B$ such that
$dist(x,y)<\eps$, and for each $y\in B$ there exists $x\in A$, such that
$dist(x,y)<\eps$. Here $dist(\cdot,\cdot)$ is the Euclidian distance.

We say that $\Omega_L$ approximates $(\Omega,h)$ as $L\to\infty$, if $\frac{1}{L}
\Omega_L \subset \Omega$ for each $L$, and for each $\eps>0$ the set $(\partial
\Omega, h)=\{(x,y)\in\mathbb R^2 \times \mathbb R\mid x\in \partial \Omega,
y=h(x)\}$ becomes within $\eps$--distance of $\left(\frac{1}{L}\partial \Omega_L,
\frac{1}{L}H_{\Omega_L}\right)$ as $L\to\infty$. Here $H_{\Omega_L}$ is the height
function of lozenge tilings of $\Omega_L$ on $\partial \Omega_L$.

\begin{theorem} \label{Theorem_LLN} Suppose that the domains $\Omega_L$ approximate $(\Omega,h)$ as
$L\to\infty$. Then the rescaled height function $\frac{1}{L} H_L(\mathbf x L,
\mathbf y L)$ of uniformly random lozenge tiling of $\Omega_L$ converges in uniform
norm, in probability to a non-random function $H_\Omega$ on $\Omega$, which
coincides with $h$ on $\partial \Omega$.
%The expectations $\E \frac{1}{L}
%H_L(\mathbf x L, \mathbf y L )$ also converge to the same limit.
\end{theorem}
 The proof of Theorem \ref{Theorem_LLN} is given in \cite{CKP}, see also \cite{CEP},
 \cite{KOS}. The function $H_\Omega$ is identified there with a solution
 to a certain variational problem and therefore depends only on $\Omega$ and $h$. More direct descriptions for restrictive classes
 of domains $\Omega$ were given in \cite{KO} and \cite{BufGor}.
 The function $H_\Omega$ is always Lipshitz, but its derivatives
 might have discontinuities, e.g.\ at the points where the inscribed circle is tangent to the
 hexagon in Figure \ref{Fig_Hex}.

\bigskip

Theorem \ref{Theorem_LLN} can be reinterpreted as the Law of Large Numbers for the
average proportions of lozenges of three types. The definition of the height
function implies that the derivatives of $H_\Omega$ in up--left and up--right grid
directions correspond to average densities of two types of lozenges:
$p^{{\scalebox{0.15}{\includegraphics{lozenge_v_down.pdf}}}}$ and
$p^{{\scalebox{0.15}{\includegraphics{lozenge_v_up.pdf}}}}$, respectively, see
Figure \ref{Fig_Hd}.
 The third density is then also reconstructed from directional derivatives,
e.g.\ using ,
$p^{{\scalebox{0.15}{\includegraphics{lozenge_hor.pdf}}}}=1-p^{{\scalebox{0.15}{\includegraphics{lozenge_v_down.pdf}}}}-
p^{{\scalebox{0.15}{\includegraphics{lozenge_v_up.pdf}}}}$. Therefore, the limit
shape $H_\Omega(x,y)$ can be encoded by $3$ functions
$p^{{\scalebox{0.15}{\includegraphics{lozenge_v_down.pdf}}}}(x,y)$,
$p^{{\scalebox{0.15}{\includegraphics{lozenge_v_up.pdf}}}}(x,y)$, and
$p^{{\scalebox{0.15}{\includegraphics{lozenge_hor.pdf}}}}(x,y)$, which sum up to
identical $1$. Then Theorem \ref{Theorem_LLN} yields that for any subdomain
$\mathcal D\subset \Omega$ the numbers of lozenges (inside $L\mathcal D$ in
uniformly random lozenge tiling of $\Omega_L$) of three types divided by $L$
converges as $L\to\infty$, in probability, to the integrals over $\mathcal D$ of
these three functions, as in Theorem \ref{Theorem_main_intro}.

\begin{figure}[h]
\begin{center}
 {\scalebox{0.9}{\includegraphics{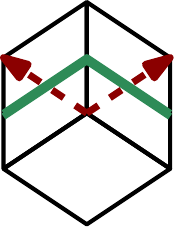}}}
\end{center}
 \caption{Derivatives in grid directions correspond to proportions of lozenges.\label{Fig_Hd}}
\end{figure}

\section{Bulk limits}

\label{Section_Main_proof}

In this section we prove the main results announced in the introduction. Theorem
\ref{Theorem_main_intro} is a combination of Theorems \ref{Theorem_main_bulk} and
\ref{Theorem_LLN_bulk}. Theorem \ref{Theorem_tilings} is a refinement of Corollary
\ref{Corollary_tilings}.

For a probability measure $\nu$ on $\mathbb R$ with bounded by $1$ density $\nu(x)$,
and two parameters $\mathbf x\in\mathbb R$, $0<\boldsymbol{\eta}<1$ we define a
function $G_{\nu,\mathbf x,\boldsymbol{\eta}}(z)$ through the formula
\eqref{eq_action_continuous}. It is an analytic function on $\mathbb C\setminus
\mathbb R$ with derivative given by
$$
 G'_{\nu,\mathbf x,\boldsymbol{\eta}}(z)= \int_{\mathbf x+\boldsymbol{\eta} -1}^{\mathbf x} \frac{dt}{z-t} -
 \int_{\mathbb R} \frac{1}{z-y} \nu(y) dy=\ln(z+1-\mathbf x-\boldsymbol{\eta})-\ln(z-\mathbf
 x)-
 \int_{\mathbb R} \frac{1}{z-y} \nu(y) dy
$$

According to Lemma \ref{Lemma_number_of_roots_cont}, the equation
$G'_{\nu,\boldsymbol{\eta}, \mathbf x}(z)=0$ has at most $1$ solution
$z=\tau(\nu,\mathbf x,\boldsymbol{\eta})$ in the upper half--plane. Denote
$$
 \xi(\nu,\mathbf x,\boldsymbol{\eta})=\frac{\tau(\nu,\mathbf x,\boldsymbol{\eta})-\mathbf x}{\tau(\nu,\mathbf x,\boldsymbol{\eta})+1-\mathbf x-\boldsymbol{\eta}},
$$
and note that when $\tau$ is in the upper half--plane, then so is $\xi$. We will say
that $\xi(\nu,\mathbf x,\boldsymbol{\eta})$ is \emph{well-defined}, if the equation
$G'_{\nu,\mathbf x,\boldsymbol{\eta}}(z)$ has a non-real solution.

\bigskip
Given a complex number $\xi$ with positive imaginary part the \emph{incomplete Beta
kernel} (see \cite{OR}, \cite{KOS}) is defined through
\begin{equation}
\label{eq_Incomplete_beta_2} K_\xi(x_1,n_1;x_2,n_2)=\frac{1}{2\pi \ii} \int_{ \bar
\xi }^{\xi} w^{x_2-x_1-1} (1-w)^{n_2-n_1}\,
  dw,
\end{equation}
where the integration contour crosses the real line inside $(0,1)$ for $n_2\ge n_1$
and inside $(-\infty,0)$ for $n_2<n_1$.

Note that when $n_1=n_2$, then we have
\begin{equation}
\label{eq_sine_kernel}
 K_\xi(x_1,n;x_2,n)=\frac{\sin(\phi(x_1-x_2)}{\pi (x_1-x_2)}, \quad \phi=\arg(\xi),
\end{equation}
which justifies the second name of $K_\xi$ which is the \emph{extended sine kernel.}

\begin{figure}[t]
\begin{center}
 {\scalebox{1.0}{\includegraphics{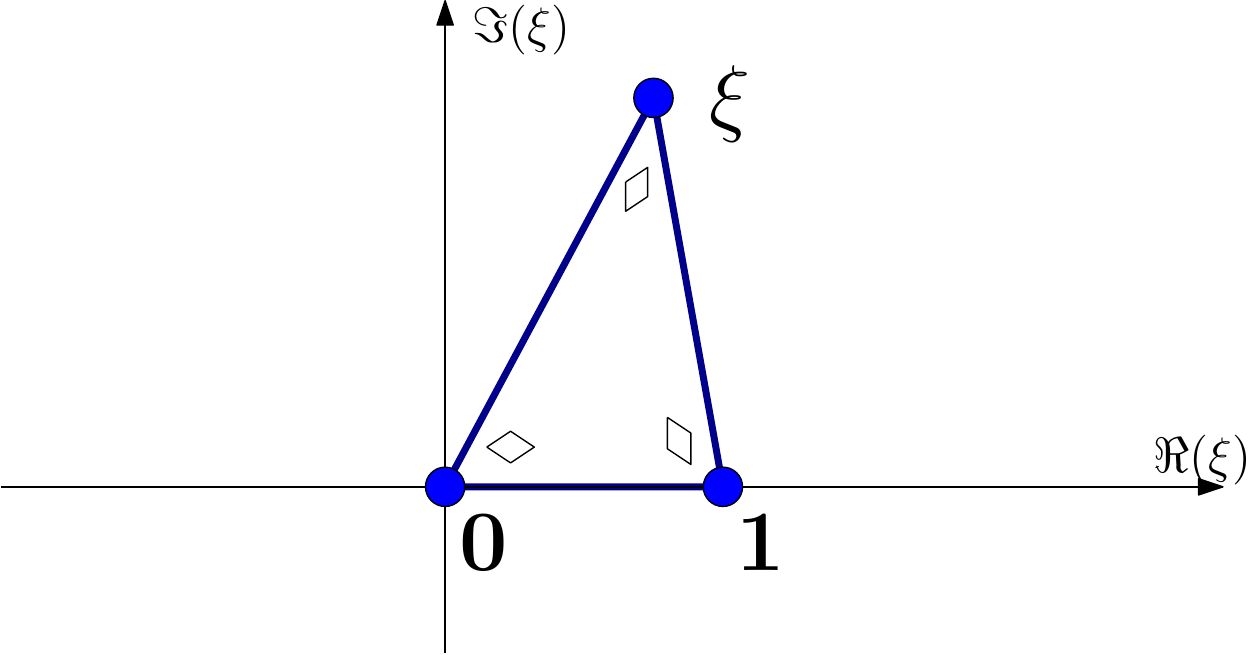}}}
\end{center}
 \caption{Correspondence between complex slope $\xi$ and local proportions of lozenges identifies with
 angles of a triangle.\label{Fig_triangle}}
\end{figure}

Further, we define $P_\xi$ as a probability measure on lozenge tilings of the plane
with correlation functions of \emph{horizontal} lozenges given for each
$k=1,2,\dots$ by:
\begin{equation} \label{eq_limit_cor}
 P_\xi[\text{there is a horizontal lozenge at }(x_i,n_i),\, 1\le i
 \le k] = \det_{i,j=1}^k [K_\xi(x_i,n_i;x_j,n_j)].
\end{equation}
It is known that $P_\xi$ is a translation invariant ergodic Gibbs measure, cf.\
\cite{Sheffield}, \cite{KOS}. We call $\xi$ the \emph{complex slope}. The
$P_\xi$--average proportions
$p^{{\scalebox{0.15}{\includegraphics{lozenge_v_down.pdf}}}}$,
$p^{{\scalebox{0.15}{\includegraphics{lozenge_v_up.pdf}}}}$, and
$p^{{\scalebox{0.15}{\includegraphics{lozenge_hor.pdf}}}}$
 of three types of lozenges
can be reconstructed through the following  geometric procedure (see \cite{KOS},
\cite{KO}): they are the angles of the triangle on $\mathbb C$ with vertices $0$,
$1$ and $\xi$ normalized to sum up to $1$, see Figure \ref{Fig_triangle}. In
particular,
$p^{{\scalebox{0.15}{\includegraphics{lozenge_hor.pdf}}}}=\frac{1}{\pi}\arg(\xi)$,
which matches \eqref{eq_sine_kernel} at $x_1=x_2$. The triplet
($p^{{\scalebox{0.15}{\includegraphics{lozenge_v_down.pdf}}}},
p^{{\scalebox{0.15}{\includegraphics{lozenge_v_up.pdf}}}},
p^{{\scalebox{0.15}{\includegraphics{lozenge_hor.pdf}}}})$ is the (geometric) slope
of $P_\xi$. \cite{Sheffield} shows that $P_\xi$ is a \emph{unique} translation
invariant ergodic measure of such slope.

The restriction of $P_\xi$ to horizontal lozenges on a vertical line is described by
the kernel \eqref{eq_sine_kernel} and has the name \emph{discrete Sine process}.

\begin{theorem} \label{Theorem_main_bulk} For each $N=1,2,\dots$, let $\t^N=(\t^N_1>\t^N_2>\dots>\t^N_N)$ be a
random $N$--tuple of integers. Suppose that:
\begin{itemize}
 \item For $\eps>0$ the random variables $\frac{1}{N}\sum_{i=1}^N\ln^{1+\eps}\left(1+|\t^N_i|/N\right)$ are tight as $N\to\infty$.
 \item The random probability measures $\mu_N=\frac{1}{N}\sum_{i=1}^N
 \delta_{\t^N_i/N}$ converge weakly, in probability to a deterministic measure
 $\mu$.
\end{itemize}
Take a sequence $\mathcal P(N)$ of probability measures on trapezoids, such that for
each $N=1,2,\dots,$ the vector $\t^N$ is $\mu_{\mathcal P(N)}$--distributed. Take
any point $(\mathbf x,\boldsymbol{\eta})$ with $0<\boldsymbol{\eta}<1$, such that
the complex number $\xi(\mu,\mathbf x,\boldsymbol{\eta})$ is well-defined. If
$n(N)$, $x(N)$ are two sequences of integers, such that $\lim_{N\to\infty}
n(N)/N=\boldsymbol{\eta}$, $\lim_{N\to\infty} x(n)/N=\mathbf x$, then the point
process of $\mathcal P(N)$--distributed lozenges near the point $(x(N), n(N))$
weakly converges as $N\to\infty$ to $P_{\xi(\mu,\mathbf x,\boldsymbol{\eta})}$
\end{theorem}
\begin{remark} \label{Remark_convergence_meaning}
 In more details, the last claim says that if
 $\rho_k^{(N)}(x_1,n_1;x_2,n_2;\dots;x_k,n_k)$ is the $k$th correlation function computing
 the probability that there is a horizontal lozenge at each of the positions
 $(x_i,n_i)$, $i=1,2\dots,k$, in the $\mathcal P(N)$--random lozenge tiling, then for each
 $k=1,2,\dots,$ and each $N$-dependent collection of integers $(x_1,n_1;x_2,n_2;\dots;x_k,n_k)$,
 such that the differences $x_i-x(N)$ and $n_i-n(N)$ do not depend on $N$,
 the limit $
  \lim\limits_{N\to\infty} \rho_k^{(N)}(x_1,n_1;x_2,n_2;\dots;x_k,n_k)$
 exists and is given by \eqref{eq_limit_cor}.
\end{remark}
\begin{proof}[Proof of Theorem \ref{Theorem_main_bulk}]
 We fix $k$ and $x_1,n_1;\dots,x_k,n_k$, and aim to compute $\lim\limits_{N\to\infty}
 \rho_k^{(N)}(x_1,n_1;x_2,n_2;\dots;x_k,n_k)$.
 For $y=(y_1>\dots,y_N)$, let $\rho_k^y(\cdot)$ denote the $k$th correlation
 function of horizontal lozenges in tilings corresponding to uniformly random
 Gelfand--Tsetlin schemes with \emph{fixed} top row $y$. Then we can write
  \begin{equation}
 \label{eq_cor_mixture}
  \rho_k^{(N)}= \E_{\tau^N}\left[ \rho_k^{\tau^N} \right].
\end{equation}
 Theorem \ref{Theorem_Leo} expresses $ \rho^y_k(x_1,n_1;\dots,x_k,n_k)$ as a determinant involving correlation kernel $K(\cdot)$ given by \eqref{eq_correlation_kernel}. The
 next step is to apply Theorem \ref{Theorem_Kernel_asympt} and we need to check that
 its assumptions are satisfied.

Choose $\delta>0$ such that $\delta$--neighborhood of $\tau(\mu,\mathbf x,\mathbf
n)$ is bounded away from the real axis. The weak convergence of $\mu_N$ towards
$\mu$ and tightness of $\frac{1}{N}\sum_{i=1}^N\ln^{1+\eps}(1+|\t^N_i|)$  imply that
for $G_i(z)$ defined through (cf.\ \eqref{eq_action_logarithm} )
$$
G_i(z)=\sum_{a=1}^{N-n(N)- n_i} \ln(z-x(N)-x_i+a) - \sum_{r=1}^N \ln(z-\t_r^N), \quad i=1,\dots,k
$$
we have $G_i'(Nz)\rightrightarrows G'(z)$ in probability, and therefore the rescaled
by $N$ critical points of $G_i$ are in $\delta$--neighborhood of $\tau(\mu,\mathbf
x,\boldsymbol{\eta})$ with probability tending to $1$ as $N\to\infty$. We conclude
that with probability tending to $1$ the assumptions of Theorem
\ref{Theorem_Kernel_asympt} are valid and we can use it to conclude that
$$
 K(n_i,x_i;n_j,x_j)=\frac{1}{2\pi \ii}  \int_{\bar \xi(N)}^{\xi(N)}
 w^{x_j-x_i-1}(1-w)^{n_j-n_i}+o(1),
$$
where $o(1)$ is a remainder which tends to $0$ in probability as $N\to\infty$, and
$$ \xi(N)= \frac{\tau_i/N-x_i/N-x(N)/N}{\tau_i/N-x_i/N-x(N)/N+1-n_i/N-n(N)/N}$$ with
$\tau_i$ being the critical point of $G_i$ in the upper half--plane.

As $N\to\infty$, $\xi(N)$ converges in probability to
$$
 \xi=\frac{\tau(\mu,\mathbf x,\boldsymbol{\eta})-\mathbf{x}}{\tau(\mu,\mathbf x,\boldsymbol{\eta})-\mathbf{x}+1-\boldsymbol{\eta}}.
$$
Since the random variable under expectation in \eqref{eq_cor_mixture} is between $0$ and $1$, the
convergence in probability implies the convergence of expectations and therefore
\begin{equation}
\label{eq_cor_converge_final}
 \lim_{N\to\infty}\rho^{(N)}_k(n_1,x_1;\dots,n_k,x_k)=\det_{i,j=1}^k \left[\frac{1}{2\pi \ii}  \int_{\bar \xi(\mu,\mathbf x,\boldsymbol{\eta})}^{\xi(\mu,\mathbf x,\boldsymbol{\eta})}
 w^{x_j-x_i-1}(1-w)^{n_j-n_i}\right]. \qedhere
\end{equation}
\end{proof}

The next step is to link the complex slope $\xi$ of Theorem \ref{Theorem_main_bulk}
to the limit shape for the height function, as discussed in Section
\ref{Section_LLN}.

\begin{theorem} \label{Theorem_LLN_bulk} For each $N=1,2,\dots$, let $\t^N=(\t^N_1>\t^N_2>\dots>\t^N_N)$ be a
random $N$--tuple of integers. Suppose that:
\begin{itemize}
 \item There exists $C>0$ such that
 $|\t^N_i|<CN$ almost surely for all $1\le i \le N$.
 \item The random probability measures $\mu_N=\frac{1}{N}\sum_{i=1}^N
 \delta_{\t^N_i/N}$ converge weakly, in probability to a deterministic measure
 $\mu$.
\end{itemize}
Take a sequence $\mathcal P(N)$ of probability measures on trapezoids, such that for
each $N=1,2,\dots,$ the vector $\t^N$ is $\mu_{\mathcal P(N)}$--distributed. Then
the rescaled height function (cf.\ Theorem \ref{Theorem_LLN} ) converges as
$N\to\infty$ uniformly, in probability to a non-random limit shape $H_\mu(\mathbf
x,\boldsymbol{\eta})$, $x\in \mathbb R$, $0 \le \boldsymbol{\eta} \le \mathbf 1$, in
the coordinate system of Figure \ref{Fig_GT_tiling}; this limit shape is encoded by
three proportions of lozenges
$({p^{{\scalebox{0.15}{\includegraphics{lozenge_v_down.pdf}}}}(\mathbf{x},\boldsymbol{\eta}),
p^{{\scalebox{0.15}{\includegraphics{lozenge_v_up.pdf}}}}(\mathbf{x},\boldsymbol{\eta}),
p^{{\scalebox{0.15}{\includegraphics{lozenge_hor.pdf}}}}}(\mathbf{x},\boldsymbol{\eta}))$.

Take any point $(\mathbf x,\boldsymbol{\eta})$ with $0<\boldsymbol{\eta}<1$, such
that the proportions of lozenges are continuous and non-zero at this point. Then the
complex number $\xi(\mu,\mathbf x,\boldsymbol{\eta})$ is well-defined, and moreover
$({p^{{\scalebox{0.15}{\includegraphics{lozenge_v_down.pdf}}}}(\mathbf{x},\boldsymbol{\eta}),
p^{{\scalebox{0.15}{\includegraphics{lozenge_v_up.pdf}}}}(\mathbf{x},\boldsymbol{\eta}),
p^{{\scalebox{0.15}{\includegraphics{lozenge_hor.pdf}}}}}(\mathbf{x},\boldsymbol{\eta}))$
are precisely the normalized angles of the triangle $0,1, \xi(\mu,\mathbf
x,\boldsymbol{\eta})$. And vice-versa, if the complex number $\xi(\mu,\mathbf
x,\boldsymbol{\eta})$ is well-defined, then the proportions of lozenges at $(\mathbf
x,\boldsymbol{\eta})$ are continuous and given by the angles of the
$0,1,\xi(\mu,\mathbf x,\boldsymbol{\eta})$ triangle.
\end{theorem}
\begin{remark} The assumption $|\t^N_i|<CN$ of Theorem \ref{Theorem_LLN_bulk} is
much more restrictive than the tightness assumption in Theorem
\ref{Theorem_main_bulk}. Probably, $|\t^N_i|<CN$ can be weakened, but we will not go
into this direction.
\end{remark}
\begin{proof}[Proof of Theorem \ref{Theorem_LLN_bulk}] The first part of the theorem, i.e.\ the convergence of the height
function to a non-random limit shape $H_\mu(\mathbf x,\boldsymbol{\eta})$ is a
direct corollary of Theorem \ref{Theorem_LLN}, as the weak
 convergence of measures $\mu_N$ implies the convergence of the height functions (which are their distribution functions)
 along the right boundary.

 For the second part, we need an explicit characterization of $H_\mu(\mathbf x,\boldsymbol{\eta})$ obtained in \cite{BufGor} and then further refined in
 \cite{BufKnizel}.

 Given a probability measure $\mu$ of bounded by $1$ density and compact support, define its exponential
 Stieltjes transform through
 $$
 E_\mu(z)=\exp\left(\int_{\mathbb R} \frac{1}{z-x} \mu(dx)\right).
$$
 $E_\mu(z)$ is an analytic function outside the support of $\mu$. Set
 \begin{equation}
 \label{eq_R}
  R_\mu(z)= E_\mu^{(-1)}(z)-\frac{z}{z-1},
 \end{equation}
where $E_\mu^{(-1)}(v)$ is the functional inverse of $E_\mu$, and $z$ is taken to be
close to $1$ in \eqref{eq_R} in order for this inverse to be uniquely defined.
$R_\mu(z)$ is a close relative of the Voiculescu $R$--transform of \cite{Voicu}, but
is slightly different, see \cite{BufGor} for the details.

\smallskip

For $\boldsymbol{\eta}\in (0,1]$, let $\mu[\boldsymbol{\eta}]$ be a  measure on
$\mathbb R$ with density at a point $x$ equal to
$p^{{\scalebox{0.15}{\includegraphics{lozenge_hor.pdf}}}}(\boldsymbol{\eta}\mathbf{x}-\boldsymbol{\eta}
,\boldsymbol{\eta} )$. Note that $\mu[\boldsymbol{\eta}]$ is a probability measure
due to combinatorial constraint on the number of horizontal lozenges on a vertical
line. Then the results of \cite[Section 3.2]{BufGor}\footnote{There is a technical
condition in \cite{BufGor}, originating from \cite{GP} that the distribution
function of $\mu$ is piecewise--continuous. However, one easily sees that this
condition is not important for the proofs. Another way to remove this condition is
to use the known continuity of dependence of the limit shape on boundary conditions,
see \cite[Proposition 20]{CEP} for the corresponding statement and proof for the
domino tilings; the proof for lozenge tilings is the same.} yield that
\begin{equation}
\label{eq_quantized_proj}
  R_{\mu[\boldsymbol{\eta}]}(v)=\frac{1}{\boldsymbol{\eta}} R_{\mu[1]}(v).
\end{equation}

Let us rewrite the equation \eqref{eq_quantized_proj} in a different equivalent
form. Let $H(z)$ denote the function $E_{\mu[1]}^{(-1)}$, then plugging
$v=E_{\mu[\boldsymbol{\eta}]}(z)$ into \eqref{eq_quantized_proj} we get
\begin{equation}
\label{eq_form_1}
 \boldsymbol{\eta} z= H(E_{\mu[\boldsymbol{\eta}]}(z)) -  \frac{E_{\mu[\boldsymbol{\eta}]}(z)}{E_{\mu[\boldsymbol{\eta}]}(z)-1} (1-\boldsymbol{\eta}).
\end{equation}
We also have
\begin{equation}
\label{eq_form_2}
 E_{\mu[1]}(H(E_{\mu[\boldsymbol{\eta}]}(z))= E_{\mu[\boldsymbol{\eta}]}(z)
\end{equation}
Express $H(E_{\mu[\boldsymbol{\eta}]}(z))$ from \eqref{eq_form_1} and plug it into
\eqref{eq_form_2} to get
\begin{equation}
\label{eq_form_3}
 E_{\mu(1)}\left(\boldsymbol{\eta} z+\frac{E_{\mu[\boldsymbol{\eta}]}(z)}{E_{\mu[\boldsymbol{\eta}]}(z)-1} (1-\boldsymbol{\eta} ) \right)= E_{\mu[\boldsymbol{\eta}]}(z).
\end{equation}
Further, let
$$w=\boldsymbol{\eta} z+\frac{E_{\mu[\boldsymbol{\eta}]}(z)}{E_{\mu[\boldsymbol{\eta}]}(z)-1} (1-\boldsymbol{\eta} ),$$
i.e.\
$$\frac{\boldsymbol{\eta} z-w}{\boldsymbol{\eta} z-w+1-\boldsymbol{\eta}}=E_{\mu[\boldsymbol{\eta}]}(z) ,$$

 then \eqref{eq_form_3} turns into
\begin{equation}
\label{eq_form_4}
 E_{\mu[1]}\left(w \right)= \frac{\boldsymbol{\eta} z-w}{\boldsymbol{\eta} z-w+1-\boldsymbol{\eta}},
\end{equation}
\begin{comment}
or, equivalently
\begin{equation}
\label{eq_form_5} w=\boldsymbol{\eta} z-\frac{E_{\mu[1]}(w)}{E_{\mu[
 1]}(w)-1} (1-\boldsymbol{\eta} ),
\end{equation}
\end{comment}
Now if we set $\boldsymbol{\eta}z - \boldsymbol{\eta}=\mathbf x$ and $w=u+1$, then
\eqref{eq_form_4} is precisely the equation $G'_{\mu,\mathbf
x,\boldsymbol{\eta}}(u)=0$.

Therefore, if the latter has only real roots, then $u$ is necessarily real, which
implies that
\begin{equation}
\lim_{v\to \mathbf x/\boldsymbol{\eta}+1} E_{\mu [\boldsymbol{\eta}]}(v) \label{eq_exp_Stieltjes_limit}
\end{equation}
 is also real. But since the density of $\mu[\boldsymbol{\eta}]$ is always between $0$ and $1$,
 this is possible only when this density at $\mathbf x/\boldsymbol{\eta}+1$ (which is
 $p^{{\scalebox{0.15}{\includegraphics{lozenge_hor.pdf}}}}(\mathbf{x},\boldsymbol{\eta})$ by the definition)  is either $0$ or $1$.

 Therefore, if at a point $(\mathbf x,\boldsymbol{\eta})$ the proportions of lozenges
 are non-zero, then $\xi(\mu,\mathbf x,\boldsymbol{\eta})$ is a well-defined complex
 number.

 It remains to prove that if $\xi(\mu,\mathbf x,\boldsymbol{\eta})$ is well-defined, then
 $({p^{{\scalebox{0.15}{\includegraphics{lozenge_v_down.pdf}}}}(\mathbf{x},\boldsymbol{\eta}),
p^{{\scalebox{0.15}{\includegraphics{lozenge_v_up.pdf}}}}(\mathbf{x},\boldsymbol{\eta}),
p^{{\scalebox{0.15}{\includegraphics{lozenge_hor.pdf}}}}}(\mathbf{x},\boldsymbol{\eta}))$
are (normalized) three angles of the triangle with vertices $0,1, \xi(\mu,\mathbf
x,\boldsymbol{\eta})$.

Note that if $\xi(\mu,\mathbf x,\boldsymbol{\eta})$ is a well-defined non-real
number, then Theorem \ref{Theorem_LLN_bulk} implies that the average proportion of
horizontal lozenges is bounded away from $0$ and $1$; this bound is uniform in a
neighborhood of $\mathbf x,\boldsymbol{\eta}$, since $\xi(\mu,\cdot,\cdot)$
continuously depends on its two last arguments near $\mathbf x,\boldsymbol{\eta}$.
Therefore, the density of $\mu [\boldsymbol{\eta}]$ is also bounded away from $0$
and $1$ in a neighborhood of $\mathbf x/\boldsymbol{\eta}+1$ and hence $\lim_{v\to
\mathbf x/\boldsymbol{\eta}+1} E_{\mu [\boldsymbol{\eta}]}(v) $ is non-real.

Recall that \eqref{eq_form_4} upon the change of variables
$\boldsymbol{\eta}z-\boldsymbol{\eta}=\mathbf x$, $w=u+1$ is precisely
$G'_{\mu,\mathbf x,\boldsymbol{\eta}}(u)=0$. Thus, comparing the expressions of
$\xi(\mu,\mathbf x,\boldsymbol{\eta})$ through $u$ ($=\tau(\mu,\mathbf
x,\boldsymbol{\eta})$) and of $E_{\mu [\boldsymbol{\eta}]}(z)$ through $w$, we
conclude that
\begin{equation} \label{eq_argument_limit}
 \lim_{v\to \mathbf x/\boldsymbol{\eta}+1} E_{\mu [\boldsymbol{\eta}]}(v)=\frac{1}{\xi(\mu,\mathbf x,\boldsymbol{\eta})},
\end{equation}
where the limit is taken from the \emph{upper} half--plane.
 On the other hand, the definition of $E_{\mu[\boldsymbol{\eta}]}$ implies
\begin{equation}
\label{eq_argument_through_density}
  \lim_{v\to \mathbf x/\boldsymbol{\eta}+1} \arg(E_{\mu [\boldsymbol{\eta}]}(v))=-\lim_{\eps\to 0} \int_{-\infty}^{\infty} \frac{\eps}{\eps^2+t^2}
p^{{\scalebox{0.15}{\includegraphics{lozenge_hor.pdf}}}}(\mathbf{x}+
t,\boldsymbol{\eta}) dt,
\end{equation}
where the first limit is taken in the \emph{upper} half--plane. Since
$f_\eps(t)=\frac{1}{\pi}\cdot \frac{\eps}{\eps^2+t^2}$ approaches the delta-function
as $\eps\to 0$ (and using also the continuity of $\xi(\mu,\mathbf
x,\boldsymbol{\eta})$ under small perturbations of $\mathbf x$), we conclude that
$$
 p^{{\scalebox{0.15}{\includegraphics{lozenge_hor.pdf}}}}(\mathbf x,\boldsymbol{\eta})=\frac{1}{\pi} \arg( \xi(\mu,\mathbf x,\boldsymbol{\eta}) ),
$$
which is precisely one of the angles of the $0,1,\xi(\mu,\mathbf
x,\boldsymbol{\eta})$ triangle. In order to identify another angle, define $\xi(\mu,
z,\boldsymbol{\eta})$
 for $z$ in the upper half--plane as an analytic continuation of $\xi(\mu,\mathbf x,\boldsymbol{\eta})$, satisfying
\begin{equation}
\label{eq_Psi_cont}
  \frac{1}{\xi(\mu, z,\boldsymbol{\eta})}=E_{\mu[\eta]}\left(\frac{z}{\boldsymbol{\eta}}+1\right).
\end{equation}

Plugging in \eqref{eq_Psi_cont} into \eqref{eq_form_3}, we get
$$
\frac{1}{ E_{\mu(1)}\left(\boldsymbol{\eta}+ z+\frac{1-\boldsymbol{\eta}
}{1-\xi(\mu, z,\boldsymbol{\eta})}  \right)}= \xi(\mu, z,\boldsymbol{\eta}).
$$
Differentiating the last equation
 with respect to $z$
and $\boldsymbol{\eta}$, one observes an identity (for all $z$ in the upper
half--plane)
$$
 \frac{\xi(\mu, z,\boldsymbol{\eta})-1}{\xi(\mu, z,\boldsymbol{\eta})} \cdot \frac{\partial}{\partial \boldsymbol{\eta}} \xi(\mu, z,\boldsymbol{\eta})=\frac{\partial}{\partial z} \xi(\mu, z,\boldsymbol{\eta}),
$$
which is (a version of) the complex inviscid Burgers' equation, cf.\ \cite{KO}.

Since the proportions of lozenges are identified with derivatives of the limiting
height function, we can further write (here $L$ is an arbitrary \emph{very large}
positive number)
\begin{multline} \label{eq_density_integration}
p^{{\scalebox{0.15}{\includegraphics{lozenge_v_down.pdf}}}}(\mathbf{x},\boldsymbol{\eta})
=\frac{\partial}{\partial \boldsymbol{\eta}} H_\mu(\mathbf
x,\boldsymbol{\eta})=-\frac{\partial}{\partial \boldsymbol{\eta}} \int_{\mathbf x}^L
\left(1-p^{{\scalebox{0.15}{\includegraphics{lozenge_hor.pdf}}}}(t,\boldsymbol{\eta})\right)
dt =
 \frac{\partial}{\partial \boldsymbol{\eta}} \int_{\mathbf x}^L
p^{{\scalebox{0.15}{\includegraphics{lozenge_hor.pdf}}}}(t,\boldsymbol{\eta}) dt\\ =
\lim_{\eps\downarrow 0} \frac{\partial}{\partial \boldsymbol{\eta}} \int_{\mathbf
x}^L
 \frac{1}{\pi} \arg (\xi(\mu,t+\ii \eps,\boldsymbol{\eta}) dt
= \frac{1}{\pi} \lim_{\eps\downarrow 0} \Im \int_{\mathbf x}^L
  \frac{\partial}{\partial \boldsymbol{\eta}} \ln \left(\xi(\mu,t+\ii \eps,\boldsymbol{\eta})\right) dt
\\
=\frac{1}{\pi} \lim_{\eps\downarrow 0} \Im \int_{\mathbf x}^L
  \frac{\frac{\partial}{\partial \boldsymbol{\eta}} \xi(\mu,t+\ii \eps,\boldsymbol{\eta})}{
  \xi(\mu,t+\ii \eps,\boldsymbol{\eta})} dt
= \frac{1}{\pi} \lim_{\eps\downarrow 0} \Im \int_{\mathbf x}^L
  \frac{1}{ \xi(\mu,t+\ii \eps,\boldsymbol{\eta})-1} \, \frac{\partial}{\partial t} \xi(\mu,t+\ii \eps,\boldsymbol{\eta})
   dt
\\=\frac{1}{\pi} \lim_{\eps\downarrow 0} \Im \Bigl[
\ln\left(1-\xi(\mu,\mathbf L+\ii
\eps,\boldsymbol{\eta})\right)-\ln\left(1-\xi(\mu,\mathbf x+\ii
\eps,\boldsymbol{\eta})\right)\Bigr] .
\end{multline}
We implicitly use in the last computation that $\xi(\mu,t+\ii
\eps,\boldsymbol{\eta})$ is non-zero due to \eqref{eq_Psi_cont}, and that it is not
equal to $1$, since its argument (computed as the integral in the right--hand side
of \eqref{eq_argument_through_density}) is strictly between $0$ and $\pi$.

 For large positive $L$, $\xi(\mu, L \eps,\boldsymbol{\eta})$ is a real number
 smaller than $1$ due to \eqref{eq_Psi_cont}. Therefore, the imaginary part of the
 first term in  the right--hand side of \eqref{eq_density_integration} vanishes.
 We conclude that
 $$
  p^{{\scalebox{0.15}{\includegraphics{lozenge_v_down.pdf}}}}(\mathbf{x},\boldsymbol{\eta})=-\frac{1}{\pi}
  \arg \bigl(1-\xi(\mu,\mathbf x ,\boldsymbol{\eta})\bigr),
 $$
 which is precisely the formula for the angle of the $(0,1,\xi(\mu,\mathbf x,\boldsymbol{\eta}))$
 triangle adjacent to the right vertex.
\end{proof}

\begin{theorem}
\label{Theorem_tilings} Consider a sequence of domains $\Omega_L$, $L=1,2,\dots,$
which approximate $(\Omega,h)$ as in Theorem \ref{Theorem_LLN}. In addition suppose
that $\Omega^{trap}_L$ is a sequence of trapezoids, such that the symmetric
difference $\Omega_L \triangle \Omega^{trap}_L$ is a union of triangles inside the
trapezoid and adjacent to its base, cf.\ Figure \ref{Fig_covered_domains}. Take a
sequence of points $x(L),n(L)$ inside $\Omega_L \cap \Omega^{trap}_L$, such that
$\lim_{L\to\infty} \frac{1}{L}(x(L),n(L))=(\mathbf x,\boldsymbol{\eta})$. If the
proportions of lozenges
$({p^{{\scalebox{0.15}{\includegraphics{lozenge_v_down.pdf}}}},
p^{{\scalebox{0.15}{\includegraphics{lozenge_v_up.pdf}}}},
p^{{\scalebox{0.15}{\includegraphics{lozenge_hor.pdf}}}}})$ encoding the limit shape
$H_\Omega$ of Theorem \ref{Theorem_LLN} are continuous and non-zero at point
$(\mathbf x,\boldsymbol{\eta})$, then the point process of lozenges near $x(L),n(L)$
in uniformly random lozenge tiling of $\Omega_L$ converges as $L\to\infty$ to a
translation invariant ergodic Gibbs measure of geometric slope
$({p^{{\scalebox{0.15}{\includegraphics{lozenge_v_down.pdf}}}}(\mathbf
x,\boldsymbol{\eta}),
p^{{\scalebox{0.15}{\includegraphics{lozenge_v_up.pdf}}}}(\mathbf
x,\boldsymbol{\eta}),
p^{{\scalebox{0.15}{\includegraphics{lozenge_hor.pdf}}}}}(\mathbf
x,\boldsymbol{\eta}))$
\end{theorem}
\begin{remark} The convergence is meant in the same sense as in Remark
\ref{Remark_convergence_meaning}.
\end{remark}
\begin{proof}[Proof of Theorem \ref{Theorem_tilings}] The proof is a combination of
Theorem \ref{Theorem_LLN} with Theorems \ref{Theorem_main_bulk} and
\ref{Theorem_LLN_bulk}. More precisely, Theorem \ref{Theorem_LLN} guarantees that
the random probability measures describing the positions of horizontal lozenges
along the boundary of $\Omega^{trap}_L$ converge as $L\to\infty$, which allows us to
use Theorem \ref{Theorem_main_bulk} to deduce the bulk limit theorem. Theorem
\ref{Theorem_LLN_bulk} then identifies the complex slope in Theorem
\ref{Theorem_main_bulk} with the complex slope of the limit shape. Since the complex
slopes are in-to-one correspondence with geometric slopes
$({p^{{\scalebox{0.15}{\includegraphics{lozenge_v_down.pdf}}}}(\mathbf
x,\boldsymbol{\eta}),
p^{{\scalebox{0.15}{\includegraphics{lozenge_v_up.pdf}}}}(\mathbf
x,\boldsymbol{\eta}),
p^{{\scalebox{0.15}{\includegraphics{lozenge_hor.pdf}}}}}(\mathbf
x,\boldsymbol{\eta}))$, the latter are identified as well.
\end{proof}

\end{document}